\documentclass[12pt,reqno]{amsart}
\usepackage{amsfonts}
\usepackage[a4paper,inner=1in,outer=1in,vmargin=1in,marginparwidth=1in]{geometry}
\usepackage{amsmath,amsthm,amsfonts,amssymb}
\usepackage{bbm}
\usepackage{cite}
\usepackage{color}
\usepackage{enumerate}

\makeatletter
\@namedef{subjclassname@2020}{\textup{2020} Mathematics Subject Classification}
\makeatother
\def\ad{\text{ad}}
\def\C{\mathbb{C}}
\def\N{\mathbb{N}}
\def\Z{\mathbb{Z}}

\newcommand{\op}{\oplus}
\newcommand{\cd}{\cdot}
\newcommand{\mf}{\mathfrak}
\newcommand{\dis}{\displaystyle}
\newcommand{\bgop}{\bigoplus}
\newcommand{\ot}{\otimes}
\newcommand{\al}{\alpha}
\newcommand{\mcal}{\mathcal}

\newcommand{\ga}{\gamma}

\newcommand{\la}{\lambda}
\newcommand{\La}{\Lambda}
\newcommand{\de}{\delta}

\newcommand{\mbf}{\mathbf}
\newcommand{\ch}{\rm{ch}}
\newtheorem{theorem}{Theorem}[section]
\newtheorem{definition}[theorem]{Definition}
\newtheorem{lemma}[theorem]{Lemma}
\newtheorem{remark}[theorem]{Remark}

\newtheorem{proposition}[theorem]{Proposition}

\newtheorem{example}[theorem]{Example}

\numberwithin{equation}{section}

\begin{document}
	\title{Category $\mcal  O$ for polynomial toroidal algebras and its subalgebras}

	\author[]{Priyanshu Chakraborty}
	\address{Priyanshu Chakraborty: Department of Mathematical Sciences, IISc, Bangalore.}
	\email{priyanshuc437@gmail.com, priyanshuc@iisc.ac.in}
	
	\date{}
	
	\maketitle
	\begin{abstract}
		In this paper we study Category $\mcal O$ for the polynomial toroidal Lie algebras and its $S,H$ type subalgebras. We classify irreducible objects of category $\mcal O$ as unique irreducble quotient of standard modules. Surprisingly, costandard objects of category $\mcal O$ arrises from Shen-Larsson type modules. We determine necessary sufficient conditions for irreducibility of Shen-Larsson modules. Finally appeling structure of Shen-Larsson modules and Soergel Tilting module theory of \cite{Soe}, we compute charcter formulas for irreducible modules and indecomposable Tilting modules  of category $\mcal O$.
	\end{abstract}
	\subsection{General notations}
	\begin{itemize}

		\item Throughout this paper we will work over the base field $\C$ of complex numbers.
		\item {  Let $\Z$,  $\mathbb N $ and $\Z_+$ be denote the set of integers, the set of non-negative integers and the set of positive integers respectively}. For $r \in \N$, $\C^r = \{(a_1, \ldots , a_r) : a_i \in \C, 1 \leq i \leq r\}$ and
		$\mathbb{R}^r, \Z^r, \N^{r}$ and $\Z_{+}^{r}$ are defined similarly.
		
		\item For any Lie algebra $\mathfrak{L}$, $U(\mathfrak{L})$ will denote universal enveloping algebra of $\mathfrak{L}$.

		\item Elements of $\C[t_1,\cdots, t_n]$ will be denoted as $t^r=t_1^{r_1} \cdots t_n^{r_n}$, where $r=(r_1,r_2,\cdots,r_n) \in \N^n$ and $t^r=0$ if $r \notin \N^n$.

	\end{itemize}
	
	\section{Introduction} 
	For a simple Lie algebra $\mf g$ an important class of modules for $\mf g$ was introduced by Bernstein-Gelfand-Gelfand in the 1970s, known as modules of BGG category $\mcal O$. This category contains many intersting modules, including all finite dimensional simple modules. Modules of this category satisfies certain finiteness conditions, and this
	category has many intersting homological properties. One of the main results in the theory of
	category $\mcal O$ was the proof of the Kazhdan–Luzstig conjecture, which gives the composition
	multiplicities of the simple modules occuring in Verma modules in terms of values of certain
	polynomials, called the Kazhdan–Luzstig polynomials. Later, category $\mcal O$
	was defined more general class of Lie algebras, namely those with a triangular
	decomposition, see \cite{RN}. Broadly, category $\mcal O$ is an interesting topic to consider for any Lie algebra and was studied by several authors for various Lie algebras, for instance see \cite{BNW,MC,MS,CT,DSY20}. In this paper we consider category $\mcal O$ for polynomial toroidal Lie algebras.\\
	Let $A_n= \C[t_1, \cdots,t_n]$ be denote the polynomial algebra in n commuting variables. Let $W_n$ be denote the Lie algebra of vector fileds on $\C^n$. It is well known that $W_n$ can be viewed as the Lie algebra of derivations over $A_n$ and this Lie algebra is called as Witt algebra. The Lie algebra $W_n$ posses three important class of subalgebras, namely subalgebra consisting of divergence zero vector fileds ($S_n$), subalgebra consisting of Hamiltonian type vector fields ($H_n$) and the subalgebra consisting of contact type vector fields ($K_n$). Represenations of $W,S, H$ type Lie algebras have been well studied in literature, for instance see \cite{Rud,Rud2,Shen1,Shen2,Shen3,DSY20}. \\
	Let $A_n^{\pm}=\C[t_1^\pm, \cdots,t_n^\pm]$ and $\rm{Der}(A_n^\pm)$ be the derivation algebra over $A_n^\pm$.  Representation for the Lie algebra $\rm{Der}(A_n^{\pm})$ and its $S,H$ type subalgebras have been well studied in literature, for instance see \cite{BT,GLZ,R04,TH,RC,FT,RP} and references therein. It has been found that there is a very close connections between modules of $W_n$ and $\rm{Der}(A_n^{\pm})$, but its still unknown how they are actually related. \\
	Let $\mf g$ be a finite dimensional simple Lie algebra. It is well known that full toroidal Lie algebras are obtained by considering semi direct product of universal central extension of $\mf g \ot A_n^{\pm}$ and $\rm{Der}(A_n^{\pm})$. Further $S,H$ type subalgebras of full toroidal Lie algebras can be described by taking $S,H$ type subalgebras of $\rm{Der}(A_n^{\pm})$ insteed of $\rm{Der}(A_n^{\pm})$.{ For more details on full toroidal Lie algebras and its $S,H$ type subalgebras see the references \cite{CLT1,CLT2,RH,RSB,TB,RC,PR,RT}.  }  It has been found that $S,H$ type subalgebras of full toroidal Lie algebras forms extended affine Lie algebras, see the references{\cite{SP1,PR,RC}} and references therein. But there are no such studies of representations for
	 the Lie algebras analogus to the full toroidal Lie algebras when one consider $A_n$ insteed to $A_n^{\pm}$. In this paper we concentrate on that problem.\\     
Let us consider the Lie algebra $\mf g \ot A_n$ and its trivial central extension $\mcal Z$ whose basis consists of vectors $\{t^rK_i: r \in \N^n, 1\leq i \leq  n  \} $. Then there is a natural action of $W_n$ on $\mf g \ot A_n \op \mcal Z$, see Section 2. Let $S_n,H_n$ be denote the subalgebras of $W_n$ as described above, for $H_n$ we consider $n=2m$. Choose $X_n$ to be a Lie algebra of the set $\{W_n,S_n,H_n\}$, then $X_n$ is a Lie subalgebra of $W_n$ and $X_n$ is $\Z$ garded.  Hence the action of $W_n$ induces an action of $X_n$ on $\mf g \ot A_n \op \mcal Z$. We consider the corresponding induced Lie algebra as $\mf L(\mf g,X_n)=(\mf g \ot A_n \op \mcal Z) \rtimes X_n$. It is a $\Z$ graded Lie algebra induced by the gradation of the polynomial Lie algebra. We study Category $\mcal O$ for the Lie algebra $\mf L(\mf g,X_n)$. \\
The paper is organized as follows. In section 2, we give the construction of the Lie algebra $\mf L(\mf g,X_n)$. We observe that $\mf L(\mf g,X_n)$ is $\Z$ graded and its zeroth graded component { $\mf L(\mf g,X_n)_0$ is $\mf g \op (X_n)_0 \op {\rm span}\{K_1, \dots, K_n\} $, where $(X_n)_0$ is the zeroth graded component of $X_n$. Further we will see that $(X_n)_0$ is a reductive Lie algebra.} In section 3, we define category $\mcal O$ for $\mf L(\mf g,X_n)$. We construct the standard module for $\mf L(\mf g,X_n)$, which is a induced module of the finite dimensional irreducible over $\mf L(\mf g,X_n)_{-1} \op \mf L(\mf g,X_n)_0 $ and denote the standard modules by $\Delta_X(\la,\mu,\mbf c)$, where $\la , \mu$ are dominant intergral weight for $\mf g$, $(X_n)_0$ respectively and $\mbf c \in \C^n$. In Proposition 3.5, we prove that irreducible objects of category $\mcal O$ are unique irreducible quotient of standard modules. In subsection 3.3, we construct Shen-Larsson type modules for $\mf L(\mf g,X_n)$. In subsection 3.4, we define costandard module $\nabla_X(\la,\mu,\mbf c)$ of category $\mcal O$, for some  dominant intergral weight $\la , \mu$ of $\mf g$ and $(X_n)_0$ respectively and $\mbf c \in \C^n$. In proposition 3.9, we prove one of the crucial result of this paper,  costandard module are isomorphic with Shen-Larsson modules. In subsection 3.6, we define excepctional Shen-Larsson modules for $\mf L(\mf g,X_n)$ and prove that excepctional modules are reducible. \\
In section 4, we consider a category of modules following \cite{Skr} called category of $(A_n, \mf L(\mf g, X_n))$ modules and use the structure of modules for this category to prove that Shen-Larsson modules are irreducible except for excecpctional modules, Theorem 4.8. In section 5, we compute the character formulas for irreducible modules of category $\mcal O$ using the stuructre costandard modules, Theorem 5.2.  In subsection 5.2, we define Tilting modules for category $\mcal O$, and prove existance of unique indecomposable tilting module for category $\mcal O$ by the use of Soergel's construction, \cite{Soe}. Finally we determine the character formulas for indecomposable tilting modules, Theorem 5.11.

	\section{Polynomial Toroidal Lie algebras and its $S,H$ type subalgebras  }
	Let $\mf g$ be a finite dimensional simple Lie algebra and $A_n=\C[t_1,\cdots, t_n]$ be the polynomial algebra over $n$ commuting variables for $n \geq 2$. Then $A_n$ is a $\Z$ graded commutative algebra with gradation induces from the degree of the polynomial. For $g=t^r$, define ${\rm deg}(g)=r_1+\cdots+r_n$. Let $W_n$ be the Lie algebra of derivations over $A_n$. It is well known that $W_n$ has a basis consisting of vectors $\{t^rd_i : r \in \N^n, \, 1 \leq i \leq n\}$, where $d_i=\frac{\partial}{\partial t_i}$ for $1 \leq i \leq n$. Clearly, $\Z$ gardation of $A_n$ induces a $\Z$ gradation on $W_n$, i.e $W_n=  \dis\bgop_{i \in \Z}{(W_n)_i}$, where ${(W_n)_i}= {\rm span}\{  gd_j: {\rm deg}(g)=i+1, g \in A_n,1 \leq j \leq n\}.$\\
	Let $S_n$ be the special type subalgebra of $W_n$ consisting of vector fields $\{\dis{\sum_{i}}g_id_i \in W_n: \dis{\sum_{i}}d_i(g_i)=0  \}.$ Then it can be seen that a spanning set for $S_n$ is given by $$S_n={\rm span}\{d_{ij}(t^r): r \in \N^n\setminus \{0\}, 1 \leq i <j \leq n    \},$$ where $d_{ij}(t^r)=d_j(t^r)d_i-d_i(t^r)d_j$.\\
	Let $H_n$ be the Hamiltonian type subalgebra of $W_n$ consisting of vector fileds which are annhilated by the 2-forms, $\dis{\sum_{i=1}^{m}}dt_i \wedge dt_{m+i} $, for $n=2m$. A spanning set for $H_n$ is given by 
	$$H_n ={\rm span}\{d_{H}(t^r): r \in \N^n\setminus \{0\}    \},$$ where
	$d_H(t^r)= \dis{\sum_{j=1}^{m}} d_j(t^r)d_{m+j} - \dis{\sum_{j=m+1}^{2m}} d_j(t^r)d_{j-m}.$ Through out the paper we fix $n=2m$ for the Lie algebra $H_n$.\\
	Let $X_n$ be denote a Lie algebra of the set $\{ W_n, S_n, H_n\}$.  Then $X_n$ is a subalgebra of $W_n$ and $\Z$ graded, induced from the gradation of $W_n$. More precisely, $(X_n)_i= \dis\bgop_{i \in \Z}{(W_n)_i} \cap X_n $ for all $i \in \Z$. In particular, we get 	$$ (X_n)_0 \cong \begin{cases}
		\mf{gl}_n, \,\, X_n =W_n,\\
		\mf{sl}_n, \,\, X_n =S_n,\\
		\mf{sp}_n, \,\, X_n =H_n.
	\end{cases}  $$ 
A clear descripction of $(X_n)_0$ given by the following.  The triangular decomposition of $(X_n)_0$ is given by $ (X_n)_0= \mf n^-_X \op \mf h_X \op \mf n^+_X$ where,
\[
\mathfrak{n}^-_X=
\begin{cases}
	\operatorname{span}\{t_id_j\mid 1\le j<i\le n\} , \text{if }X_n=W_n,S_n,\\
	\operatorname{span}\{t_id_j-t_{j+m}d_{i+m},\,t_{s+m}d_t+t_{k+m}d_s\mid
	1\le j<i\le m,\,1\le s\le k\le m\}, \text{if }X_n=H_n.
\end{cases}
\] 
\[
\mathfrak{h}_X=
\begin{cases}
	\operatorname{span}\{t_i\partial_i\mid 1\le i\le n\}, \text{if } X_n=W_n,\\
	\operatorname{span}\{t_i\partial_i-t_j\partial_j\mid 1\le i<j\le n\}, \text{if } X_n=S_n,\\
	\operatorname{span}\{t_id_i-t_{i+m}d_{i+m}\mid 1\le i\le m\}, \text{if }X_n=H_n.
\end{cases}
\]
\[
\mathfrak{n}^+_X=
\begin{cases}
	\operatorname{span}\{t_id_j\mid 1\le i<j\le n\} , \text{if }X_n=W_n,S_n,\\
	\operatorname{span}\{t_id_j-t_{j+m}d_{i+m},\,t_{s+m}d_t+t_{k+m}d_s\mid
	1\le i<j\le m,\,1\le k\le s\le m\}, \text{if }X_n=H_n.
\end{cases}
\] 
Now it is easy to see that the isomorphism of $(X_n)_0$ is obtained by sending $t_id_j \to E_{ij}$, for $1 \leq i,j \leq n$.\\
Now we are ready to define the polynomial toroidal Lie algebra and its $S,H$ type subalgebras.
Let us consider the Lie algebra $\mf g \ot A_n$ and its trivial central extension $\mcal Z = {\rm span}\{ t^rK_i : r \in \N^n, 1 \leq i \leq n\}$.
Now	consider the Lie algebra $\mf L(\mf g,W_n)=(\mf g \ot A_n \op \mcal Z)\rtimes W_n$, bracket operations on $\mf L(\mf g,W_n)$ is given by
	\[   [x\ot t^r,y\ot t^s]  =[x,y]\ot t^{r+s},    	\]
	\[    [t^rK_i,x\ot t^s]=[t^rK_i, t^sK_j]=0,\]
	\[	[t^rd_i, t^sK_j]=s_it^{r+s-e_i}K_j,\]
	\[	[t^rd_i, x\ot t^s]=s_ix\ot t^{r+s-e_i},\]
	\[  [t^rd_i,t^sd_j]=s_it^{r+s-e_i}-r_jt^{r+s-e_j}, \]
	for all $r,s \in \N^n,x,y \in \mf g, \, 1\leq i,j \leq n. $ It can be easily seen that degree gradation of $A_n$ induces a gradation on $\mf L(\mf g,W_n)$, which makes it a $\Z$-graded Lie algebra, i.e $\mf L(\mf g,W_n) =\dis\bgop_{i \in \Z}{\mf L(\mf g,W_n)_i}$, where ${\mf L(\mf g,W_n)_i}= {\rm span}\{ \mf g \ot f, \,fK_j,\, \, gd_j: deg(f)=i, \, deg(g)=i+1, f,g \in A_n,1 \leq j \leq n\}.$ We call $\mf L(\mf g, W_n)$ as polynomial toroidal Lie algebra. \\ 
	  Note that $X_n$ is a subalgebra of $W_n$, hence we have an induced action of $X_n$ on $\mf g \ot A_n \op \mcal Z$. We denote the induced Lie algebra as 
	$$ \mf L(\mf g,X_n)=(\mf g \ot A_n \op \mcal Z)\rtimes X_n . $$
	It is clear that $\Z$-gradation of $\mf L(\mf g,X_n)$ is given by  $\mf L(\mf g,X_n) =\dis\bgop_{i \in \Z}{\mf L(\mf g,W_n)_i}\cap \mf L(\mf g,X_n)$.\\
Now, we have the following description for graded components of $\mf L(\mf g, X_n)$.
	$$ {\mf L(\mf g,X_n)_{-1}} \cong {\rm span}\{d_1,\cdots,d_n\} \cong \C^n,  $$
	$$    {\mf L(\mf g,X_n)_{0}} \cong 
	\mf g \op (X_n)_0  \op {\rm span}\{K_1, \cdots, K_n\}, \,\,  \text{and} $$
	$${\mf L(\mf g,X_n)_i}=0, \,\, \forall \, i <-1.$$ It is clear that $\mf L(\mf g,X_n)_0$ is a reductive Lie algebra. Further one can see that center of $\mf L(\mf g,X_n)={\rm span}\{K_1, \cdots, K_n\} $. We denote the center by $\mcal Z (\mf L(\mf g,X_n))$. 
	Let $\mf H=\mf h \op \mf h_X  $, be the Cartan subalgebra of $\mf L(\mf g,X_n)_0$, where $\mf h$ is the Cartan subalgebra of $\mf g$ and $\mf h_X$ is that of $(X_n)_0$. Note that $\Z$ gradation of $\mf L(\mf g,X_n)$ induces a $\Z$ gradation on $U(\mf L(\mf g,X_n))$, i.e $U(\mf L(\mf g,X_n)) = \dis{\bgop_{i \in \Z}} U(\mf L(\mf g,X_n))_i$ and ${\rm dim} U(\mf L(\mf g,X_n))_i < \infty$.
	\section{Category $\mcal O$ }
	\subsection{Category $\mcal O$ for $\mf L(\mf g,X_n)$} The following notion of category $\mcal O$ is an analogy of category for semi-simple Lie algebras, \cite{Hum}. An interested reader can see \cite{DSY20} for the definition of Category $\mcal O$ for $X_n$.
	\begin{definition}\label{category O}
		Denote by $\mathcal{O}$ the BGG category of $\mf L(\mf g, X_n)$-modules, whose objects $M$ are additive groups and satisfy the following properties.
		\begin{itemize}
			\item[(1)] $M$ is an admissible $\mathbb{Z}$-graded $\mf L(\mf g,X_n)$-module, i.e., $M=\bigoplus\limits_{i\in\mathbb{Z}} M_{i}$ with $\dim M_{i}<+\infty$, and
			$\mf L(\mf g,X_n)_{i}M_{j}\subseteq M_{i+j},\forall\,i,j \in \Z$.
			\item[(2)] $M$ is locally finite for ${\rm P}_X= \mf L(\mf g,X_n)_{-1} \op \mf L(\mf g,X_n)_0 $.
			{ \item[(3)] $M$ is $\mf H$-semisimple, i.e., $M$ is a weight module: $M=\bigoplus_{\lambda\in\mf H^*}M_{\lambda}$.}
			
		\end{itemize}
		The morphisms in $\mathcal{O}$ are $\mf L(\mf g,X_n)$-module morphisms that respect the $\mathbb{Z}$-gradation i.e.,
		$${\rm Hom}_{\mathcal{O}}(M, N)=\{f\in{\rm Hom}_{U(\mf L(\mf g,X_n))}(M, N)\mid f(M_{i})\subseteq N_{i},\,\\ \,\forall\,i\in\mathbb{Z}\},\, \forall\,M,N\in\mathcal{O}.$$
		
	\end{definition}
	Additionally, we say an object $M \in \mcal O$ has depth $d$ if $M =\bigoplus\limits_{i \geq d} M_{i}$	with $M_d \neq 0$ and $M_i=0$, for all $i < d$. The following Lemma follows from the fact that $(X_n)_0$ action on $\C^n$ is irreducible, which is a standard fact.
	\begin{lemma}\label{lem X_n acts irr}
	$(X_n)_0$ action on $\mf L(\mf g, X_n)_{-1}$ is irreducible.	
	\end{lemma}
	\begin{lemma}\label{Trivial action of L_{-1}}
		Let $V$ be a finite dimensional irreducible module for $\mf L(\mf g,X_n)_{-1}\op \mf L(\mf g,X_n)_{0}$. Then $\mf L(\mf g,X_n)_{-1}$ acts trivially on $V$.
	\end{lemma}
	\begin{proof}
		Let $W=\{v \in V: \mf L(\mf g,X_n)_{-1}.v=0\}$. At first we show that $W$ is a $\mf L(\mf g,X_n)_{-1}\op \mf L(\mf g,X_n)_{0}$ submodule of $V$.
		Note that for all $x \in \mf L(\mf g,X_n)_{-1} $ and $y \in \mf L(\mf g,X_n)_{-1}\op \mf L(\mf g,X_n)_{0}$ we have,  $x.(y.v)=y.x.v+[x,y].v =[x,y].v=0,  $ since $[x,y] \in \mf L(\mf g,X_n)_{-1}$. This proves that $W$ is a $\mf L(\mf g,X_n)_{-1}\op\mf L( \mf g,X_n)_0$ submodule.\\
		Hence to complete the proof, it is sufficient to show that $W \neq 0$. Since $\mf L(\mf g,X_n)_{-1}$ is an abelian subalgebra of $\mf L(\mf g,X_n)_{-1}\op \mf L(\mf g,X_n)_{0}$, hence there exists a non-zero $v \in V$ such that $x.v=\la(x)v$ for all $x \in \mf L(\mf g,X_n)_{-1}$. Now $V$ is irreducible, hence we must have $V=U( \mf L(\mf g,X_n)_{0})v$. Therefore there exists $Y_1, \dots, Y_n \in U( \mf L(\mf g,X_n)_{0})$ such that $\mcal B=\{v, Y_1.v, \dots , Y_n.v \}$ forms a basis for $V$. Now we prove that for any $x \in \mf L(\mf g,X_n)_{-1}$ the matrix of $x$ with respect to $\mcal B$ forms a upper triangular matrix with all diagonal entry $\la(x)$. This will follow from the following observation.
		$$ x.(Y_i.v)=\la(x)Y_i.v+[x,Y_i].v =\la([x,Y_i])v+\la(x)Y_i.v,  \,\, \text{since} \, [x,Y_i] \, \in U(\mf L(\mf g)_{-1}), \, \forall \, i .$$
		By Lemma \ref{lem X_n acts irr},  $\mf L(\mf g,X_n)_{-1}$ is an irreducible representation of $ \mf L(\mf g,X_n)_{0}$. Hence for non-zero $x,y \in \mf L(\mf g,X_n)_{-1}$ there exists a $Y \in U( \mf L(\mf g,X_n)_{0})$ such that $[Y,y]=x$. Now considering trace on both side we conclude that trace of the matrix of $x$ is zero, i.e $\la(x)=0$. This proves that $W \neq 0$.
		
	\end{proof}

	\begin{remark}
		\begin{itemize}
			\item[(1)] It is easily seen that $\mcal O$ is an abelian category
			\item[(2)] Let $\mcal O_P^{lf}$ be denote the category of locally finite $P$ modules. Let $M\in \mcal O_P^{lf}$, then $M$ is finite dimensional. Hence from Lemma \ref{Trivial action of L_{-1}}, it follows that $\mf L(\mf g,X_n)_{-1}$ acts trivially on $M$.
		\end{itemize}
	\end{remark}
	\subsection{Standard objects in $\mcal O$}\label{sec: standard obj}
	\subsubsection{}\label{basic notations}
	Let $\mf g=\mf n^- \op \mf h \op \mf n^+$ and $(X_n)_0=\mf n^-_X \op \mf h_X \op \mf n^+_X$ be the triangular decomposition of $\mf g$ and $(X_n)_0$ respectively. Let $\Lambda^+_{\mf g,X}$ be the set of dominant integral weights of $\mf g \op (X_n)_0$ relative to the standard Borel subalgebra $\mathfrak{b}=(\mf h \op\mf n^+) \op (\mf h_X \op \mf n^+_X) $ of $\mf L(\mf g,X_n)_{0}$, i.e.
	$$\La^+_{\mf g, X}=\{ (\la,\mu): \la \in \La_{\mf g}^+, \mu \in \La_{X}^+\},$$  where $\La_{\mf g}^+$ is the set of dominant integral weights for $\mf g$ and $\La_X^+$ is the set of dominant integral weights for $(X_n)_0.$ Denote $	\La^+_{\mf g, X,\C^n}=\La^+_{\mf g, X}\times \C^n.$
	
	It is well-known that finite-dimensional irreducible $\mf L(\mf g,X_n)_0$-modules are isomorphic to $V(\la)\ot L^0(\mu)$ for some irreducible $\mf g$ module $V(\la)$ and irreducible $(X_n)_0$ module $L^0(\mu)$ on which $K_i$ acts by scalar $c_i \in \C$. We denote these modules by $L^0(\la,\mu, \mbf c)$, for some $\mbf c=(c_1, \cdots, c_n) \in \C^n$. Further it is known that $\Z$-graded finite-dimensional irreducible $\mf L(\mf g, X_n)_0$-modules are parameterized by ${\Lambda^+_{\mf g,X,\C^n}}\times\mathbb{Z}$. For any $(\lambda, \mu, \mbf c)\in {\Lambda^+_{\mf g, X,\C^n}}$ define ${}^dL^0_X(\lambda,\mu,\mbf c) $ as the irreducible  $\mf L(\mf g,X_n)_0$-module concentrated in a single degree $d$ isomorphic with $V(\la)\ot L^0(\mu)$.  Set $${}^d\Delta_X(\lambda, \mu,\mbf c)=U(\mf L(\mf g,X_n))\otimes_{U({\rm P}_X)}{}^dL^0_X(\lambda,\mu,\mbf c),$$
	where ${}^dL^0_X(\lambda,\mu,\mbf c)$ is regarded as a ${\rm P}_X$-module with trivial $\mf L(\mf g,X_n)_{{-1}}$-action. Then ${}^d\Delta_X(\lambda,\mu,\mbf c)$ is called a standard module of depth $d$ and
	$\{{}^d\Delta_X(\lambda,\mu,\mbf c)\mid (\lambda, \mu,\mbf c)\in {\Lambda^+_{\mf g,X,\C^n}}\,,d\in\mathbb{Z}\}$ constitute a class of standard modules of depth $d$ for $\mf L(\mf g,X_n)$ in the usual sense.  The following result is analogus to \cite[Lemma 3.1]{DSY20}.
	{
		\begin{proposition}\label{lem1}
			Let $(\lambda,\mu,\mbf c)\in{\Lambda^+_{\mf g, X,\C^n}}\, , d\in\mathbb{Z}$. The following statements hold.
			\begin{itemize}
				\item[(1)] The standard module ${}^d\Delta_X(\lambda,\mu,\mbf c)$ is an object in $\mathcal{O}$.
				\item[(2)] The standard module ${}^d\Delta_X(\lambda,\mu,\mbf c)$ has a unique irreducible quotient, denoted by ${}^dL_X(\lambda,\mu,\mbf c)$.
				\item[(3)] The iso-classes of irreducible modules in $\mathcal{O}$ are parameterized by ${\Lambda^+_{\mf g,X,\C^n}}\times \mathbb{Z}$. More precisely, each irreducible module $S$ in $\mathcal{O}$ is of the form $L_X(\la,\mu, \mbf c)$ for some $(\la,\mu,\mbf c)\in \La^+_{\mf g,X,\C^n}$ with the depth $d$.
			\end{itemize}
		\end{proposition}
	}
\begin{proof}
(1)	Note that as a vector space, ${}^d\Delta_X(\la,\mu,\mbf c) \cong U(\mf L(\mf g,X_n)_{\geq 1})\ot_\C {}^dL^0_X(\la,\mu,\mbf c) = \dis{\bgop_{i \in \Z_{\geq 0}}}M_i$, where $M_i=  U(\mf L(\mf g,X_n)_{\geq 1})_i\ot_\C {}^dL^0_X(\la,\mu,\mbf c) $ and $ U(\mf L(\mf g,X_n)_{\geq 1})_i$ is the $i$-th garded component of  $U(\mf L(\mf g,X_n)_{\geq 1}).$ Clearly it is $\mf H$ semi-simple. Thus to complete the proof of (1) we need to show that ${}^d\Delta_X(\la,\mu,\mbf c) $ is $U({\rm P}_X)$ locally finite. Let us choose a typical element as $w=\dis{\sum_{s=1}^{k}} w_s \ot v$, where $v \in {}^dL^0_X (\la,\mu,\mbf c)$ and $w_s \in U(\mf L(\mf g,X_n)_{\geq 1})_s.$ 
	  $$U({\rm P}_X) \cd w \subseteq \dis{\sum_{s=1}^{k}} U({\rm P}_X)w_s \ot {}^dL^0_X(\la,\mu,\mbf c)  \subseteq  \dis{\sum_{s=0}^{k}} U(\mf L(\mf g,X_n)_{\geq 1})_s \ot {}^d L^0_X(\la,\mu,\mbf c) < \infty  .$$
	  This completes the proof of (1).\\
(2) Let $W$ be any non zero proper submodule of $	 {}^d\Delta_X(\la,\mu,\mbf c) $. Then $W \cap M_0$ is a $(\mf L(\mf g, X_n))_0$ submodule of $M_0$. But $M_0$ is irreducible $(\mf L(\mf g, X_n))_0$ module. Hence $W \cap M_0= 0$ or $M_0$. If $W\cap M_0 \neq 0$ , then $W = {}^d\Delta_X(\la,\mu,\mbf c)$. Thus  $W \cap M_0= 0$, i.e $W \subseteq \dis{\bgop_{i \in \Z_{> 0}}}M_i$. Hence sum of all proper submodule is the maximal proper submodule of ${}^d\Delta_X(\la,\mu,\mbf c) $. Therefore standard module has a unique irreducble quotient.\\
(3) Let $S$ be any irreducble module in $\mcal O$. $U({\rm P}_X) \cd v$  finite dimensional $U({\rm P}_X)$ module for all $v \in S$. Choose an irreducible finite dimensional $U({\rm P}_X)$ submodule $S_1$ of  $U({\rm P}_X) \cd v$ for some $v \in S$. By Remark 2.4, $\mf L(\mf g, X_n)_{-1}$ acts trivially on $S_1$. Therefore $S_1$ is finite dimensional $\mf L(\mf g, X_n)_{0}$ module isomorphic to ${}^dL^0_X(\la, \mu , \mbf c)$ for some $(\la,\mu, \mbf c )\in \La^+_{\mf g,X} \times \C^n$ and $d \in \Z$. Consequently $S$ is an irreducble quotient of $ {}^d \Delta_X(\la,\mu, \mbf c)$ isomorphic to  ${}^dL_X(\lambda,\mu,\mbf c)$. Note that the depth d of ${}^dL_X(\lambda,\mu,\mbf c)$ is determinded by the action of $U({\rm P}_X) \cd v$ , since $S=U(\mf L(\mf g,X_n)) \cd v =U(\mf L(\mf g,X_n)_{\geq 1}) U({\rm P}_X) \cd v$ and $U(\mf L(\mf g,X_n)_{\geq 1})$ can not decrease the depth. This completes the proof.

\end{proof}
\subsubsection{Shift functors by shifting depthes}\label{sec: sub O}
	Set $$\mathcal{O}_{\geq d}:=\{M\in\mathcal{O}\mid M=\sum\limits_{i\geq d}M_{i}\},$$ which consists of objects admitting depths not smaller than $d$.
	Consider the shift functor  $T_{d',d}:\mathcal{O}_{\geq d}\longrightarrow\mathcal{O}_{\geq d'}$ which maps the object  $M=\sum_{k}M_k\in \mathcal{O}_{\geq d}$   to $M[d'-d]\in\mathcal{O}_{\geq d'}$, where $M[d'-d] =\sum_{k}(M[d'-d])_k$ and $M[d'-d]_k=M_{k+d'-d}$.  The following fact is clear.
	
	\begin{lemma}\label{lem: shift functor}
		The functor $T_{d',d}$ induces a category equivalence.
	\end{lemma}
	
	With the shift functors, we can focus our concern on $\mathcal{O}_{\geq 0}$. Therefore from this point we only consider the depth of the modules as zero.
	\subsection{Shen-Larsson module for $\mf L(\mf g,X_n)$} We define Shen-Larsen type module structure on $A_n \ot V$, for a $\mf L(\mf g,X_n)_0$  module $V$. Main ideas behind these module structures comes from the Shen-Larsen modules of $X_n$. An interested reader can see \cite{Shen1,Shen2,Shen3,Rud,Rud2,DSY20} for more details on Shen-Larsen modules over the Lie algebra $X_n$. Further, in section 4, we see another descripction of these module structures which will be more convinient to understand these module structures.

\subsubsection{\bf $\mf L(\mf g,W_n)$ module structure:} A $\mf L(\mf g,W_n)$ module structure on $A_n \ot V$ is given by:
	\begin{align}\label{rho module W_n}
		t^rd_i \cdot (t^s\ot v)=t^rd_i(t^s)\ot v+ \dis{\sum_{j=1}^{n}}d_j(t^r)t^s\ot x_jd_i\cd v\\
		x\ot t^r \cdot (t^s\ot v)=t^{r+s}\ot x \cd v\\
		t^rK_i \cd (t^s\ot v) =t^{r+s}\ot K_i \cd v,
	\end{align}
	for all $r,s\in \N^n, v \in V, 1 \leq i \leq n, x \in \mf g.$ 
\subsubsection{\bf $\mf L(\mf g,S_n)$ module structure:} A $\mf L(\mf g,S_n)$ module structure on $A_n \ot V$ is given by:
\begin{align}\label{rho module S_n}
	d_{ij}(t^r) \cdot (t^s\ot v)=	d_{ij}(t^r)t^s\ot v 
	+ d_id_j(t^r)t^s\ot (x_id_i-x_jd_j)\cd v     \cr
	 +\dis{\sum_{\substack{k=1\\ k\neq i}}^{n}}d_kd_j(t^r)t^s\ot x_kd_i \cd v -  \dis{\sum_{\substack{k=1\\ k\neq j}}^{n}}d_kd_i(t^r)t^s\ot x_kd_j \cd v \\
	x\ot t^r \cdot (t^s\ot v)=t^{r+s}\ot x \cd v\\
	t^rK_i \cd (t^s\ot v) =t^{r+s}\ot K_i \cd v,
\end{align}
for all $r,s\in \N^n, v \in V, 1 \leq i<j \leq n, x \in \mf g.$ 

\subsubsection{\bf $\mf L(\mf g,H_n)$ module structure:} A $\mf L(\mf g,H_n)$ module structure on $A_n \ot V$ is given by:
\begin{align}\label{rho module H_n}
	d_{H}(t^r) \cdot (t^s\ot v)=	d_{H}(t^r)t^s\ot v 
	+\dis{\sum_{\substack{k=1}}^{m}}d_k^2(t^r)t^s\ot x_kd_{k+m}  \cd v -	\dis{\sum_{\substack{k=m+1}}^{2m}}d_k^2(t^r)t^s\ot x_{k}d_{k-m} \cd v \cr
     -\dis{\sum_{k=1}^{m}}\dis{\sum_{j=m+1}^{2m}}d_jd_k(t^r)t^s\ot (x_kd_{j-m} -x_{j}d_{m+k}) \cd v +\dis{\sum_{1 \leq j <k\leq m}}d_jd_k(t^r)t^s\ot (x_kd_{m+j} +x_{j}d_{m+k}) \cd v    \cr
       \dis{\sum_{m+1 \leq j <k\leq 2m}}d_jd_k(t^r)t^s\ot (x_kd_{j-m} +x_{j}d_{k-m}) \cd v  \\
	x\ot t^r \cdot (t^s\ot v)=t^{r+s}\ot x \cd v\\
	t^rK_i \cd (t^s\ot v) =t^{r+s}\ot K_i \cd v,
\end{align}
for all $r,s\in \N^n, v \in V, 1 \leq i\leq n, x \in \mf g.$ It is routine to check that the actions defined in subsections (3.3.1), (3.3.2) and (3.3.3) are module actions on $A_n \ot V$ for $\mf L(\mf g, X_n)$. We denote these representations as $(A_n \ot V, \rho_X)$.

	\begin{remark}
		Let $V \cong V'$ be two isomorphic $\mf L(\mf g,X_n)_0$ modules. Then $A_n \ot V$ and $A_n \ot V'$ are isomorphic $\mf L(\mf g,X_n)$ modules. In particular, it can be verified that if $\phi: V \to V'$ be the isomorphism of $\mf L(\mf g,X_n)_0$ modules, then 
		$$\psi: A_n \ot V \to A_n \ot V' \,\, \text{by}   $$
		$$t^r \ot v \to t^r \ot \phi(v), $$
		$ \forall \, r \in \N^n, v \in V $ is an isomorphism of $\mf L(\mf g,X_n)$ modules. 
	\end{remark}
	
		\begin{lemma}\label{Lemma contain V }
			Let $V$ be an irreducible $\mf L(\mf g,X_n)_0$ module.	Then any non-zero submodule of $A_n \ot V$ contains $1 \ot V$. 
		\end{lemma}
		\begin{proof}
		By applying actions of $d_i$  sufficiently many times we can conclude that every non-zero submodule contains a vector of the form $1 \ot v$, for some $v \in V$. Now by irreducibility of $V$ we have the result.  
		\end{proof}
	
	\subsection{Costandard objects in $\mcal O$ and its realization}\label{sec: costandard obj}
	
	\subsubsection{Costandard modules} For $(\lambda,\mu,\mbf c)\in{\Lambda^+_{\mf g,X,\C^n}}$ , define the costandard $\mf L(\mf g,X_n)$-module corresponding to $(\lambda, \mu,\mbf c) $ as $$\nabla_X(\lambda,\mu,\mbf c):={\rm Hom}_{U(\mf L(\mf g,X_n)_{\geq0})}(U(
	\mf L(\mf g,X_n)), L^0_X(\lambda, \mu,\mbf c)),$$
	where $L^0_X(\lambda,\mu,\mbf c)$ is regarded as a $\mf L(\mf g,X_n)_{0}$-module with trivial $\mf L(\mf g,X_n)_{\geq 1}$-action. The module structure on $\nabla_X(\lambda,\mu,\mbf c)$ is given by:
	$$ (X.\phi)(Y)=\phi(YX)  \, \,\,\,  \forall \, X \in \mf L(\mf g,X_n), \, Y \in U(\mf L(\mf g,X_n)), \, \phi \in \nabla_X(\la,\mu,\mbf c).  $$
	From the next subsection (Proposition \ref{lem: simple socles}) it will clearly follow that $\nabla_X(\lambda,\mu,\mbf c) \in\mcal O$.

	\subsection{Realizations of costandard modules}	In this subsection, we introduce a concrete realization of costandard modules $\nabla_X(\lambda,\mu, \mbf c)$ for $(\lambda,\mu,\mbf c)\in{\Lambda^+_{\mf g, X,\C^n}}$.  For $(\la,\mu,\mbf c) \in {\Lambda^+_{\mf g, X,\C^n}}$ consider the irreducible $\mf L(\mf g,X)_0$-module  $L^0_X(\la,\mu,\mbf c)$ and set $$\mathcal{V}_{X}(\lambda,\mu,\mbf c)=A_n\otimes L^0_X(\lambda,\mu,\mbf c),$$
	then $\mathcal{V}_{X}(\lambda,\mu,\mbf c)$ is a $\mf L(\mf g,X_n)$ module under the actions defined in subsections, (3.3.1), (3.3.2) and (3.3.3).

	\begin{proposition}\label{pro: Costandard mod rel}
		$	\nabla_X(\lambda,\mu,\mbf c) \cong \mathcal{V}_X(\lambda,\mu, \mbf c)$ as $\mf L(\mf g,X_n)$-module.
	\end{proposition}
	The followings are preparation for the proof of  Proposition \ref{pro: Costandard mod rel}. Let us define  $\psi:\mathcal{V}_X(\lambda,\mu,\mbf c) \to 	\nabla_X(\lambda,\mu, \mbf c)$  by:
	$$  t^r  \ot v \ot w \mapsto \psi(t^r\ot v \ot w), \, \forall r \in \N^n, \, v \in V(\la), \, w \in L^0(\mu), $$
	where $$\psi(t^r\ot v \ot w)(d^a)=d^a(t^r)|_{(t_1,\cdots,t_n)=(0,\cdots,0)} v \ot w , \, \forall \, a \in \N^n,$$
	
	and $d^a=d_1^{a_1}\cdots d_n^{a_n}$. Note that 	$\psi(t^r\ot v \ot w)$ is well defined due to the fact 
	
	$$\nabla_X(\lambda,\mu,\mbf c)=	{\rm Hom}_{U(\mf L(\mf g,X_n)_{\geq0})}(U(\mf L(\mf g,X_n) ),L^0_X(\lambda,\mu,\mbf c)) \cong {\rm Hom}_{\C}	(U(\mf L(\mf g,X_n)_{-1}),  L^0_X(\lambda,\mu,\mbf c )).$$
	
	Now we prove that $\psi$ is a $\mf L(\mf g,X_n)$ module map. For that we record the following lemma which is easy to verify by induction.
	\begin{lemma}\label{comp formula} The following formulas holds:
		$$d^r(x\ot t^s)=\sum_{0\preceq\kappa\preceq r} C_\kappa^r P_\kappa^s x\ot t^{s -\kappa} d ^{r-\kappa},$$
		$$d^r( t^s K_j)=\sum_{0\preceq\kappa\preceq r} C_\kappa^r P_\kappa^s  t^{s -\kappa}K_j d ^{r-\kappa}, \, 1 \leq j \leq n,$$
		
		for all $r, s \in \N^n, \, x \in \mf g$.
	\end{lemma}
	
	\begin{lemma}\label{Lg module map}
		$\psi$ is a $\mf L(\mf g,X_n)$ module map.
	\end{lemma}	
	\begin{proof}
	We prove this Lemma for $X_n=W_n$, other types follows with similar lines. Note that for all $r,s, \al \in \N^n$, $v \in V(\la)$ and $w \in L^0(\mu)$ we have, 
		\begin{align*}
			\psi(x\ot t^r\cd t^s\ot v \ot w)(d^\al) &= \psi( t^{r+s}\ot x\cd v \ot w)(d^\al) \cr
			&=d^\al(t^{r+s})|_{(t_1,\cdots,t_n)=(0,\cdots,0)}x\cd v\ot w \cr
			&=\de_{\al,r+s}(r+s)!x \cd v\ot w.
		\end{align*}
		Again, 
		\begin{align*}
			(x\ot t^r\cd \psi(t^s\ot v \ot w))(d^\al) &=  \psi(t^s\ot v \ot w)(d^\al \cd x \ot t^r) \cr
			&=  \psi(t^s\ot v \ot w)(\sum_{0\preceq\kappa\preceq\alpha} C_\kappa^\alpha P_\kappa^r x\ot t^{r -\kappa} d ^{\alpha-\kappa}) (\text{by Lemma \ref{comp formula}  }) \cr
			&= \sum_{0\preceq\kappa\preceq\alpha} C_\kappa^\alpha P_\kappa^r x\ot t^{r -\kappa} \cd (\psi(t^s\ot v \ot w)d ^{\alpha-\kappa}) \cr
			&=  C_r^\alpha P_r^r x\ot 1 \cd (d^{\al-r}(t^s)|_{(t_1,\cdots,t_n)=(0,\cdots,0)}v\ot w) \cr
			&= \de_{\al, r+s}C_r^{r+s} P_r^r x\ot 1 \cd (s!v \ot w) \cr
			&=\de_{\al, r+s}(r+s)!x \cd v\ot w ,
		\end{align*}	
		note that above equalities holds due to the fact $\psi(t^s\ot v \ot w)$ is a $U(\mf L(\mf g,X_n)_{\geq 0})$ module map and $U(\mf L(\mf g,X_n)_{\geq 1})$ acts trivially on $L^0_X(\la,\mu,\mbf c)$ and hence we get $	\psi(x\ot t^r\cd t^s\ot v \ot w)=x\ot t^r \cd 	\psi( t^s\ot v \ot w)$, for all	$r,s \in \N^n$, $v \in V(\la)$ and $w \in L^0(\mu).$ One can follow the same steps to show that $	\psi( t^rK_j\cd t^s\ot v \ot w)= t^rK_j \cd 	\psi( t^s\ot v \ot w)$, for all $1 \leq j \leq n$,	$r,s \in \N^n$, $v \in V(\la)$ and $w \in L^0(\mu).$
		Also following the computation of \cite[Proposition 3.4]{DSY20}, we can conclude that $\psi( t^rd_i \cd  t^s\ot v \ot w)= t^r d_i  \cd 	\psi( t^s\ot v \ot w)$, for all	$r,s \in \N^n$, $v \in V(\la)$ and $w \in L^0(\mu)$, which proves that $\psi$ is a $\mf L(\mf g,W_n)$ module map. 
		
	\end{proof}
	{\bf Proof of Proposition 2.10:}	
	With the help of Lemma \ref{Lg module map}, it is sufficient to prove that  $\psi$ is a bijection. Note that as a vector space $\nabla_X(\la,\mu,\mbf c) \cong {\rm Hom}_{\C}	(U(\mf L(\mf g,X_n)_{-1}),  L^0_X(\lambda,\mu,\mbf c )) \cong {\rm Hom}_{\C}	(U(\mf L(\mf g,X_n)_{-1}), \C) \ot  L^0_X(\lambda,\mu ,\mbf c) $ and hence $\nabla_X(\la,\mu,\mbf c)$ is spanned by $\{f_{q,v,w}:q=(q_1,\dots,q_n)\in \N^n,\ v \in V(\la), \, w \in L^0(\mu)  \},$ where $	f_{q,v,\mu}(d^\al)$ is defined by \\
	$$	f_{q,v,w}(d^\al)=\de_{\al,q}v\ot w , \ \al \in \N^n,v \in V(\la), \, w \in L^0(\mu).  $$
	This proves that $\psi$ is surjective since $cf_{q,v,\mu}$ is the image of $t^q \ot v\ot w$ for some non-zero scalar $c$. To prove $\psi$ is injective consider the relation $\psi(\dis{\sum_{i=1}^{k}} f_i \ot v_i )=0 , \ f_i \in A_n , \ v_i \in L^0(\la,\mu,\mbf c)$ and $f_i'$s are homogeneous. Then we operate $\psi(\dis{\sum_{i=1}^{k}} f_i \ot v_i )$ on $d^\al$ for various $\al \in \N^n$, which concludes that $v_i=0$ for all $i$. This completes the proof.
	
		{

		{\begin{proposition}\label{description of socle}\label{lem: simple socles}
				The costandard module $\nabla_X(\la,\mu,\mbf c)$ are objects of category $\mcal O$. Further $\nabla_X(\la,\mu,\mbf c)$ admits a simple socle which is isomorphic to $L_X(\la,\mu,\mbf c)$.
			\end{proposition}
			
			\begin{proof}
				By proposition \ref{pro: Costandard mod rel}, $\nabla_X(\la,\mu,\mbf c) \cong \mcal V_X(\la,\mu,\mbf c),$ which clearly implies that $\nabla_X(\la,\mu,\mbf c)$ $ \in  \mcal O$.  Now by Lemma \ref{Lemma contain V }, every non-zero submodule of $\mcal V_X(\la,\mu,\mbf c)$ contains $L^0_X(\la,\mu,\mbf c)$. Let $S$ be any irreducble submodule of $\mcal V_X(\la,\mu,\mbf c)$. We claim that $S =U(\mf L(\mf g, X_n))(L^0_X(\la,\mu,\mbf c))$. First note that $U(\mf L(\mf g, X_n))$ $(L^0_X(\la,\mu,\mbf c))$ is irreducible, since if  $W$ be a non zero submodule of  $U(\mf L(\mf g, X_n))(L^0_X(\la,\mu,\mbf c))$, then $W$ is a non-zero submodule of  $\mcal V_X(\la,\mu,\mbf c)$, hence contains $L^0_X(\la,\mu,\mbf c)$, so $W=U(\mf L(\mf g, X_n))(L^0_X(\la,\mu,\mbf c))$. Also it is clear that any non-zero submodule contains $U(\mf L(\mf g, X_n))$ $(L^0_X(\la,\mu,\mbf c))$, so the claim follows. In particular, $\mcal V_X(\la,\mu,\mbf c)$ contains only one irreducible submodule. Thus to complete the proof we need to show that $S \cong L_X(\la,\mu,\mbf c)$. For this one need to observe that $\nabla_X(\la,\mu,\mbf c) \to S$, $u \ot v \mapsto u \cd v$ define a surjective module map, for all $u \in U(\mf L(\mf g,X_n)), v \in L^0_X(\la,\mu,\mbf c).$ Now the isomorphism follows from  Proposition \ref{lem1}.
			\end{proof}
		}
	}
	\subsection{Exceptional costandard modules for $\mf L(\mf g,X_n)$}
Denote $n_X=n, n-1$ or $m$, according as $X_n=W_n,S_n $ or $H_n$. We call $\mcal V_X(0,\mu_k,\mbf 0)$ as exceptional modules for $\mf L(\mf g,X_n)$, where
	
$$ 	\mu_k=\begin{cases}
		\epsilon_1+\cdots+\epsilon_k, \ \ \text{for }\, X_n=W_n. \,\, \\
		\epsilon_1+\cdots+\epsilon_k, \ \ \text{for }\, X_n=S_n. \,\,\\
		\epsilon_1+\cdots+\epsilon_k, \ \ \text{for }\, X_n=H_n \,\,  ,\\
	\end{cases} 
	$$
for all $ 1 \leq k \leq n_X$, where	
	  $\epsilon_j:\mf h_W \to \C$ is the linear map $\epsilon_j(t_id_i)=\de_{ij}$, for all $1 \leq i,j \leq n$ and $\mu_0=0$. We call the tuple $(0,\mu_k,\mbf 0)$, for $0 \leq k \leq n_X$ as exceptional weights for $\mf L(\mf g,X_n)$ and the corresponding modules as exceptional modules. 
	 
\subsubsection{\bf Exceptional modules for $W_n$}	Note that for $X_n=W_n$ with  $(\la,\mu, \mbf c)=(0,\mu_k,\mbf 0)$ with $0\leq k \leq n$,  $L^0_W(0,\mu_k, \mbf 0)$ has a realization: 
	$$ L^0_W(0,\mu_k, \mbf 0)\cong \C\ot \wedge^k\C^n.     $$
	Hence from this point we treat $\mcal V_W(0, \mu_k, \mbf 0)$ as $A_n \ot \C \ot \wedge^k\C^n,$ for $0 \leq k \leq n$.

	Now we can define a de Rahm complex of $\mf L(\mf g,W_n)$ modules $\mcal V(0, \mu_k,\mbf 0)$, given by:
	\begin{align*}\label{eq: type I complex}
		\mathcal{V}(0,\mu_0,\mbf 0)\xrightarrow{\,\,\,\mf d_0\,\,\,}\mathcal{V}(0,\mu_1, \mbf 0)\xrightarrow{\,\,\,\mf d_1\,\,\,}\cdots\cdots \mathcal{V}(0,\mu_k,\mbf 0)\xrightarrow{\,\,\,\mf d_k\,\,\,}\mathcal{V}(0,\mu_{k+1}, \mbf 0)\xrightarrow{\mf d_{k+1}}\cdots\cdots \\ \mathcal{V}(0,\mu_{n-1}, \mbf 0) \xrightarrow{\mf d_{n-1}}
		\mathcal{V}(0,\mu_{n}, \mbf 0)\longrightarrow 0,
	\end{align*}   
	
	where $\mf d_k: \mathcal{V}(0,\mu_{k},\mbf 0) \to \mathcal{V}(0,\mu_{k+1},\mbf 0)$ is defined by:
	
	$$\mf d_k:\mathcal{V}(0,\mu_k,\mbf 0)\longrightarrow \mathcal{V}(0,\mu_{k+1},\mbf 0)$$
	\[t^{r}\otimes 1 \ot (e_{j_1}\wedge\cdots\wedge e_{j_k})\longmapsto\sum\limits_{i=1}^nd_i(t^{r})\otimes 1 \ot (e_{j_1}\wedge\cdots\wedge e_{j_k}\wedge e_i) ,\]
	
	where $e_{1},\cdots, e_{n}$ is the basis of $\C^n$ and $r\in {\mathbb{N}}^n.$ One should observe that $\mcal V(0, \mu_k, \mbf 0)$  are basically modules for the Witt algebra $W_n$, since $\mf g$ and $K_i$ acts trivially on $L^0(0,\mu_k,\mbf 0)$. Now it follows from \cite{Shen1,Shen2,Shen3,DSY20, Rud} that $\mf d_k$ is a $\mf L(\mf g,W_n)$ module map and further these maps are exact. We have the following proposition. 
	\begin{proposition}\label{pro V(0,mu,0) red}
		
		\begin{itemize}

			\item[(1)]	$\mcal V_W(0,\mu_k, \mbf 0)$ is irreducible if and only if $k=n$.
			
			\item[(2)] For $0 \leq k \leq n-1$, $\mcal V_W(0,\mu_k,\mbf 0)$ has two composition factors isomorphic to $L_W(0,\mu_k,\mbf 0)$ and $L_W(0,\mu_{k+1},\mbf 0)$ with free multiplicity. Moreover, $L_W(0,\mu_0,\mbf 0)$ is one dimensional trivial module.

		\end{itemize}
	\end{proposition}
	\begin{proof}
		Note that (1) and (2) follows from [\cite{Shen3}, Theorem 2.4]. (3) follows from Proposition \ref{description of socle} with the help of (2).
	\end{proof}

\subsubsection{\bf Exceptional modules for $S_n$ and $H_n$} Following subsection 2.6.1 and  \cite[Theorem 3.1, Theorem 3.2]{Shen3},  \cite[Theorem 4.8, Theorem 5.10]{Rud2} and \cite[Theorem 4.5, Theorem 4.6]{DSY20} we can conclude the following two proposition for $\mf L(\mf g, S_n)$ modules. 

	\begin{proposition}\label{pro V(0,mu,0) S}
	
	\begin{itemize}

		\item[(1)] For $0 \leq k \leq n-1$, $\mcal V_S(0,\mu_k,\mbf 0)$ has two composition factors isomorphic to $L_S(0,\mu_k,\mbf 0)$ and $L_S(0,\mu_{k+1},\mbf 0)$  with free multiplicity. Moreover, $L_S(0,\mu_0,\mbf 0)$ is one dimensional trivial module.
		
	\end{itemize}
\end{proposition}

	\begin{proposition}\label{pro V(0,mu,0) H}
	
	\begin{itemize}
		\item[(1)] For $0 \leq k \leq m$, $\mcal V_H(0,\mu_k,\mbf 0)$ has composition factors isomorphic to $L_H(0,\mu_{k-1},\mbf 0)$, $L_H(0,\mu_k,\mbf 0)$,  $L_H(0,\mu_{k+1},\mbf 0)$ with $[\mcal V_H(0,\mu_k,\mbf 0):L_H(0,\mu_{k-1},\mbf 0)] = [\mcal V_H(0,\mu_k,\mbf 0):L_H(0,\mu_{k+1},\mbf 0)]=1$ and $[\mcal V_H(0,\mu_k,\mbf 0):L_H(0,\mu_{k},\mbf 0)]$ $=2$, where $L_H(0,\mu_{-1},\mbf 0)=L_H(0,\mu_{m+1},\mbf 0)= 0 $ and $L_H(0,\mu_{0},\mbf 0)$ is one dimensional trivial module.
		
	\end{itemize}
\end{proposition}

	\section{Category of $(A_n, \mf L(\mf g,X_n))$ modules}
	This section is devoted to prove that $\mcal V_X(\la,\mu,\mbf c)$ is reducible only when modules are Exceptional. For this we define a new category of modules, called $(A_n,\mf L(\mf g, X_n))$-mod and use some property for objects of this category to conclude our results. To define this category we follow \cite{Skr,YB,YB2}, (also see \cite{CSS} for Lie super algebra $W(m,n)$) where the author defined this kind of category of modules for the Lie algebra $X_n$ over arbitrary characteristic base field. 
\subsection{Category of $(A_n, \mf L(\mf g, X_n)$ modules):}	\begin{definition}\label{Def A, L(g) mod}
		Denote by $(A_n,\mf L(\mf g,W_n))$-mod the category whose objects are additive groups $M$ endowed with a $A_n$-module structure $(.)_{A_n}$, a $\mf L(\mf g,X_n)$-module structure $\rho_X$  and a $\mf L(\mf g,X_n)_{\geq 0}$-module structure $\sigma_X$ so that the following properties are satisfied:
		\begin{itemize}
			\item[(I)] (i) $[\rho_X(Y),f_{A_n}]=(Y(f))_{A_n}$, $\forall, f \in A_n$ and $ Y \in X_n.$\\
			(ii)  $[\rho_X(x),  f_{A_n}]=0$, $\forall \, x \in \mf g \ot A_n \op \mcal Z, f \in A_n.$ 
			
			\item[(II)] $[\sigma_X(D'),f_{A_n}]=0$, $\forall \,  D' \in \mf L(\mf g,X_n)_{\geq 0}, f \in A_n. $
			\item[(III)] $[\rho_X(d_i),\sigma_X(D')]= 0$, $ \forall \,  D' \in \mf L(\mf g,X_n)_{\geq 0}$ and $1 \leq i \leq n$.
			\item[(IV)] (i) 
		$	\begin{cases}
				\rho_W(fd_j)=f_{A_n} \circ \rho_W(d_j) +\dis\sum_{i=1 }^{n}(d_i f)_{A_n}\circ \sigma_W (t_i d_j) , \forall \, 1 \leq i \leq n, f\in A_n. \\
				\rho_S(d_{ij}(f))= (d_jf)_{A_n}\circ \rho_S(d_i) -(d_if)_{A_n}\circ \rho_S(d_j) + \dis{\sum_{|\al|=2}} (d^{\al}f)_{A_n} \circ \sigma_S(d_{ij}(t^{\al})), \cr \,	\hspace{9.5cm}  \forall \, \leq i < j \leq n, f\in A_n.\\
				\rho_H(d_{H}(f)) = \dis{\sum_{j=1}^{m}}(d_jf)_{A_n} \circ \rho_H(d_{m+j}) -\dis{\sum_{j=1}^{m}}(d_jf)_{A_n} \circ \rho_H(d_{j-m}) \cr
			\hspace{2.5cm}	+ \dis{\sum_{|\al|=2}} (d^{\al}f)_{A_n} \circ \sigma_H(d_{H}(t^{\al}))
			\end{cases}
			 $\\
			(ii)  $ \rho_X(x\ot f)=\sigma_X(x) \circ f_{A_n}, $ $\forall \, x \in \mf g, f \in A_n.$\\
			(iii)  $ \rho_X( fK_j)=\sigma_X(K_j) \circ f_{A_n}, $ $\forall \, 1\leq j \leq n, \, f \in A_n.$
		\end{itemize}
		
		We call this category as $(A_n, \mf L(\mf g,X_n))$ module category. This category will be denoted by $(A_n, \mf L(\mf g))$-mod, whose objects are usually called $(A_n,\mf L(\mf g,X_n))$-modules. For notational convenience some times we use $\sigma$ and $\rho$ instead to $\sigma_X$ and $\rho_X$ respectively, whenever $X_n$ is clear from the context.
	\end{definition}
	\begin{example}\label{Ex A,L mod}
		Let us consider the vector space $M=A_n \ot V$, for some ${\mf L(\mf g,X_n)_0}$ module $V$. Now define the $A_n$ modules structure on $M$ as polynomial multiplication, { $\rho_X$ module structure as subsections (2.3.1), (2.3.2) and (2.3.2)}. Further define the trivial module action of $\mf L(\mf g,X_n)_{\geq 1}$ on $V$ and consider it as a module for $\mf L(\mf g,X_n)_{\geq 0}$. Now define a $\sigma_X$ module structure on $M$ by:
		\begin{align}\label{sigma action}
			y. (f \ot v)=f \ot y \cd v , \,\,\forall \, y \in \mf L(\mf g,X_n)_{\geq 0}, v \in V, f \in A_n.
		\end{align}
		It is easy to see that under the above actions $M$ satisfied all the properties of $(A_n, \mf L(\mf g,X_n))$ modules. Hence $M$ is an object of $(A_n, \mf L(\mf g,X_n))$-mod. In particular, when $V=L^0_X(\la,\mu, \mbf c)$, then $M=\mcal V_X(\la,\mu,\mbf c)$.  \qed
	\end{example}
	We record a result from \cite[Proposition 3.4]{Skr}  in our context which we will use in this section. For more details on the following proposition one can see \cite[Section 3]{Skr}.
	\begin{proposition}\label{prop of f , g exists}
		Given any $\ga=(\ga_1,\cdots, \ga_n) \in \N^n,$ there exists $k \in \N$ and $f_1, \cdots, f_k;$ $g_1,\cdots, g_k$ in $A_n$ such that $\dis{\sum_{i=1}^{k}}f_id^\al g_i=  
		\begin{cases}
			1 , \,\,\text{if} \, \al=\ga\\
			0, \,\, \text{if} \, \al \neq \ga
		\end{cases} $
	\end{proposition}
	\begin{example}
		We give an example to demonstrate Proposition \ref{prop of f , g exists}. Consider $n=2$ and $\ga=(1,1)$. Now set $f_1=1,g_1=t_1t_2$; $f_2=-t_2,g_2=t_1$; $f_3=-t_1,g_3=t_2$ and $f_4=t_1t_2, g_4=1$. It can be easily seen that this set of polynomials satisfy the property of Proposition \ref{prop of f , g exists}. 
	\end{example}
	\begin{lemma}\label{Lemma A, L(g) submod}
		Let $M$ be an  $(A_n, \mf L(\mf g,X_n))$-module such that $M$ is annihilated by $\sigma_X|_{\mf L(\mf g,X_n)_{\geq 1}}$. Then every  $A_n$ and $\rho_X$-submodule  $M'$ of $M$ is $(A_n, \mf L(\mf g,X_n))$-submodule.
	\end{lemma}	
	
	\begin{proof}
		We need to show that $M'$ is a $\sigma_X(\mf L(\mf g,X_n))$-submodule of $M$. Note that by the assumption we only need to show $M'$ is invariant under $\sigma_X|_{(\mf L (\mf g,X_n)_{0}}$. We prove this case by case.\\
		{\bf Case I:} Let $X_n=W_n$. Choose $\ga=e_u \in \N^n$, then there exists $f_1, \cdots, f_k$ and $g_1, \cdots g_k$ such that Proposition \ref{prop of f , g exists} holds for $e_u$. Now we have,
	\begin{align}\label{eqn W}
		\dis{\sum_{r=1}^{k}}f_r \rho_W( g_rd_j) &= \dis{\sum_{r=1}^{k}}f_r (g_{r} \circ \rho_W(d_j) +\dis\sum_{i=1 }^{n}(d_i g_r)\circ \sigma_W (t_i d_j)) \cr &=\sigma_W(t_ud_j),
	\end{align} 
		since $\dis{\sum_{i=1}^{k}}f_id^\al g_i=0$ for all $\al \neq e_u$. \\
		{\bf Case II:} Let $X_n=S_n.$ Choose, $\ga= e_i+e_j$, which satisfy the property of  Proposition \ref{prop of f , g exists}. Now we have,
		\begin{align}\label{eqn S}
		\dis{\sum_{r=1}^{k}}f_r \rho_S( d_{ij}(g_r)) &= \dis{\sum_{r=1}^{k}}f_r ((d_jg_r)_{A_n}\circ \rho_S(d_i) -(d_ig_r)_{A_n}\circ \rho_S(d_j) + \dis{\sum_{|\al|=2}} (d^{\al}g_r)_{A_n} \circ \sigma_S(d_{ij}(t^{\al}))) \cr
		&= \dis{\sum_{r=1}^{k}}\dis{\sum_{|\al|=2}} f_r (d^{\al}g_r)_{A_n} \circ \sigma_S(d_{ij}(t^{\al}))) \, \, (\text{By Proposition \ref{prop of f , g exists}} )\cr
		&= \begin{cases}
			\sigma_S(t_id_i-t_jd_j), \,\,\, \text{if }\gamma= e_i+e_j, \, 1 \leq i <j \leq n.\\
			\sigma_S(t_kd_i), \,\,\, \text{if }\gamma= e_j+e_k, \, k \neq i, i<j, 1 \leq i,j,k \leq n.
		\end{cases}
			\end{align}

{\bf Case III:} Let $X_n=H_n.$ Choose, $\ga= e_i+e_j$, which satisfy the property of  Proposition \ref{prop of f , g exists}. Now we have,		
	\begin{align}\label{eqn H}
		\dis{\sum_{r=1}^{k}}f_r \rho_H( d_{H}(g_r)) &=  \dis{\sum_{r=1}^{k}}f_r (\dis{\sum_{j=1}^{m}}(d_jg_r)_{A_n} \circ \rho_H(d_{m+j}) -\dis{\sum_{j=1}^{m}}(d_jg_r)_{A_n} \circ \rho_H(d_{j-m}) \cr
		&+ \dis{\sum_{|\al|=2}} (d^{\al}g_r)_{A_n} \circ \sigma_H(d_{H}(t^{\al}))) \, \, (\text{By Proposition \ref{prop of f , g exists}}) \cr
		&= \dis{\sum_{r=1}^{k}}\dis{\sum_{|\al|=2}}f_r (d^{\al}g_r)_{A_n} \circ \sigma_H(d_{H}(t^{\al}))) \cr
		&=\begin{cases}
			2\sigma_H(t_id_{m+i}), \,\,\, \text{if }\gamma= e_i+e_i, \, 1 \leq i \leq m.\\
			-2\sigma_H(t_id_{i-m
			}), \,\,\, \text{if }\gamma= e_i+e_i, \, m+1 \leq i \leq 2m.\\
		\sigma_H(t_jd_{m+i}-t_id_{j-m}
		), \,\,\, \text{if }\gamma= e_j+e_i, \, 1 \leq i \leq m; \, \\
		\hspace{5cm} m+1 \leq j \leq 2m.\\
			\sigma_H(t_jd_{m+i}+t_id_{j+m}
		), \,\,\, \text{if }\gamma= e_j+e_i, \, 1 \leq i<j \leq m.\\
			-\sigma_H(t_jd_{i-m}+t_id_{j-m}
		), \,\,\, \text{if }\gamma= e_j+e_i, \, m+1 \leq i<j \leq 2m.
		\end{cases}
	\end{align}	
		
From the equations  (\ref{eqn W}), (\ref{eqn S}) and (\ref{eqn H}) we get that $\sigma_X|_{(X_n)_0}$ leaves $M'$ invariant, since LHS of above equations leaves $M'$ invariant by the hypothesis. Again consider the following 
	$$ \rho_X(x\ot 1)=\sigma_X(x) \circ 1,  \, \forall \, x \in \mf g, \,\, \, \text{and}$$
	$$ \rho_X(K_j)=\sigma_X(K_j) \circ 1,  \, \forall \, 1\leq j \leq n, $$
	which implies that $\sigma_X|_{\mf g \op \mcal Z(\mf L(\mf g))}$ leaves $M'$ invariant. This completes the proof.\\	
		
	\end{proof}
\begin{remark}
	From the proof of Lemma $4.5$, one can easily observe that Definition $4.1(IV)$ actually gives the Shen-Larsson module structute on $A_n \ot V$ as $\mf L(\mf g,X_n)$ modules, for a $\mf L(\mf g,X_n)_0$ module $V$.
\end{remark}

	Before going to the main theorem of this section we record an auxiliary lemma regarding representation of ${\mf {gl}}_n$ on an irreducible module $V$.
	\begin{lemma}\label{Lemma auxiliary}
		Let $(V,\eta)$ be an irreducible module for ${\mf {gl}_n}$ which satisfied the property:\\
		\begin{itemize}
			\item[(1)] $\de_{rs}\eta(E_{s'r'}) +\de_{rs'}\eta(E_{sr'}) - \eta(E_{sr}) \eta(E_{s'r'}) - \eta(E_{s'r})\eta(E_{sr'}) =0, \, \text{if} \, s \neq s',$  $\forall \, 1 \leq r,s,r',s' \leq n$.
			\item[(2)] $ \de_{rs}\eta(E_{sr'}) - \eta(E_{sr}) \eta(E_{sr'}) =0$, $\forall \, 1 \leq r,s,r',s' \leq n$.
			 \end{itemize} 
	Then $V$ is an irreducible module with highest weight $\mu_k$ for some $0 \leq k \leq n$.
	\end{lemma}
	\begin{proof}
		One can proceed similarly like [\cite{Skr}, Lemma 4.3]. Just note that relations are different in our case which leads the weights as highest weight $\mu_k$ (insteed of lowest weight of [\cite{Skr}, Lemma 4.3]).
	\end{proof}

	\begin{theorem}\label{Thm for irr}
		Let $M$ be a $(A_n, \mf L(\mf g,X_n))$-module such that  $M$ is annihilated by $\sigma_X|_{{\mf L(\mf g,X_n)}_{\geq 1}}$. Additionally,  assume that $M$ is completely reducible as $\sigma_X|_{{\mf L(\mf g,X_n)}_{ 0}}$-module and none of its irreducible components are isomorphic with $L^0_X(0, \mu_k,\mbf 0)$, $0 \leq k \leq n_X$. Then any $\rho_X$ submodule $M'$ of $M$ is $(A_n, \mf L(\mf g,X_n))$-submodule.
		
	\end{theorem}
	\begin{proof}
		We will follow arguments of \cite{Skr} to prove this theorem. In view of Lemma \ref{Lemma A, L(g) submod} it is sufficient to that any $\rho_X$-submodule $M'$ is $A_n$-submodule. Let $M_1=\{m \in M: A_n \cd m \subseteq M'\}$ be the largest $A_n$-submodule contained in $M'$ and $M_2=A_n \cd M'$ be the smallest $A_n$-submodule containing $M'$. It follows from Definition \ref{Def A, L(g) mod} (I) that $M_1,M_2$ are $\mf L(\mf g,X_n)$-submodules and hence $(A_n, \mf L(\mf g,X_n) )$-submodules, by Lemma \ref{Lemma A, L(g) submod}. Therefore we can construct a $\mf L(\mf g,X_n)$-submodule $M'/M_1$ contained inside $M_2/M_1$. Then our proof for pass from $M'$ to the $\mf L(\mf g,X_n)$ submodule $M'/M_1$ of the $(A_n, \mf L(\mf g,X_n) )$ module $M_2/M_1$. Thus we can assume that $M'$ contains no non-zero $A_n$-submodules of $M$ and $A_n .M'=M$. Then to complete the proof we need to show that $M=0$. Assume that $M \neq 0$. Let $S$ and $S'$ be two associative algebras generated by $\{\sigma_X(Y): Y \in {\mf L(\mf g,X_n)}_{\geq 0}\}$ and $\{ \rho_X(Y): Y \in \mf L(\mf g,X_n) \}$ respectively.  Now we look for endomorphisms $\phi$ of $M$ lying in $S$ with the property that for any $f \in A_n$ the endomorphism $f \phi \in S'.$ Then $\mf L(\mf g,X_n)$-submdoule $M'$ is stable under $f \phi$  for any $f \in A_n$, since $f \phi \in S'$ and $M'$ is $\rho$ submodule. Hence it contains the $A_n$-submodule $A_n \phi(M')$. Now by hypothesis $\phi(M')=0$. Again from Definition \ref{Def A, L(g) mod} (II), $\phi$ is an $A_n$-module homomorphism, since $\phi \in S$. Hence we have $$\phi(M)=\phi(A_n\cd M')=A_n \cd \phi(M')=0,$$ i.e $\phi =0$. This gives us certain relation between endomorphisms of $\sigma_X(Y), Y \in \mf L(\mf g,X_n)_{\geq 0}$.\\
		Now we compute the restrictions on the action of $\sigma_X(\mf L(\mf g,X_n)_{\geq 0})$, which turns out $\mf L(\mf g,X_n)_0$ modules to be exceptional modules. We complete the proof of this theorem with the following arguments.\\
			First note that, $\rho_X(x\ot f)=\sigma_X(x)\circ f_{A_n}=	f_{A_n}\sigma_X(x)$, for all $x \in \mf g, f \in A_n$, since $f_{A_n} $ and $\sigma_X(x)$ are commuting. Similarly, $\rho_X(fK_j)=\sigma_X(K_j)\circ f_{A_n}=	f_{A_n}\sigma_X(K_j),$ for all $f \in A_n , 1\leq j \leq n$. Therefore from the above discussion we have $\sigma_X|_{\mf g \op \mcal Z(\mf L(\mf g, X_n))}=0.$ Thus to complete the proof we need to show that $\sigma_X|_{(X_n)_0}$ is an irreducble module with highest weight $\mu_k$ for some $0 \leq k \leq n_X$. Then we have $\sigma_X|_{(X_n)_0}$ must have an irreducble compotent isomorphic with $L^0_X(0, \mu_k,\mbf 0)$ for some $\mu_k$ with $0 \leq k \leq n_X$, a contradiction. \\
	{\bf Case I:} Let $X_n =W_n$.	For any given $s, s' \in \{1,2,\cdots,n\}$, set $\ga=\epsilon_s+\epsilon_{s'}$. Then there exists $f_1, \cdots, f_k$ and $g_1, \cdots , g_k$ such that Proposition \ref{prop of f , g exists} holds. For any $f \in A_n$ we have,
		\begin{align}\label{key eqn 1}
			\dis{\sum_{i=1}^{k}}d_r(ff_i)g_i=  d_r(\dis{\sum_{i=1}^{k}}ff_ig_i)- \dis{\sum_{i=1}^{k}}ff_id_r(g_i) =0 , \, \forall \, 1 \leq r\leq n.
		\end{align}    
		
		Also we have
		\begin{align}\label{key eqn 2}
			\dis{\sum_{i=1}^{k}} d_{r}(ff_{i})d_{r'}(g_{i}) &= d_r(\dis{\sum_{i=1}^{k}} (ff_{i})d_{r'}(g_{i})) -\dis{\sum_{i=1}^{k}}ff_id_rd_{r'}(g_i) \cr
			&= \begin{cases}
				0 , \,\,\, \,\,\,\,\, \text{if,} \, \{r,r'\}\neq  \{s,s'\}\\
				-f , \,\,\, \text{if,} \, \{r,r'\}= \{s,s'\}.
			\end{cases}
		\end{align}
		Above two equations holds by Proposition \ref{prop of f , g exists}. Now consider the following action:	
		
		\begin{align*}
			\rho(ff_{i}d_{r})\rho(g_{i}d_{r'}) &= \left(ff_{i}\rho(d_{r})+\sum\limits^{n}_{j=1}d_{j}
			(ff_{i})\sigma(E_{jr}) \right) \left(   g_{i}\rho(d_{r'})+\sum\limits^{n}_{j'=1}d_{j'}
			(g_{i})\sigma(E_{j'r'})  \right) \cr
			&= ff_i\rho(d_{r})g_i\rho(d_{r'}) + \sum\limits^{n}_{j'=1}ff_i\rho(d_r) d_{j'} (g_{i})\sigma(E_{j'r'}) \cr
			& + \sum\limits^{n}_{j=1}d_{j}
			(ff_{i})\sigma(E_{jr})g_i\rho(d_{r'}) +\dis{\sum_{j=1}^{n} \sum_{j'=1}^{n}} d_{j}
			(ff_{i})d_{j'}
			(g_{i})\sigma(E_{jr})\sigma(E_{j'r'}) \cr
			&= ff_ig_i \rho(d_{r})\rho(d_{r'}) +ff_id_r(g_i)\rho(d_{r'}) + \sum\limits^{n}_{j'=1}ff_id_r d_{j'} (g_{i})\sigma(E_{j'r'}) \cr
			&+ \sum\limits^{n}_{j'=1}ff_i  d_{j'} (g_{i})\rho(d_r)\sigma(E_{j'r'}) + \sum\limits^{n}_{j=1}d_{j}
			(ff_{i})g_i\sigma(E_{jr})\rho(d_{r'}) \cr
			& +\dis{\sum_{j=1}^{n} \sum_{j'=1}^{n}} d_{j}
			(ff_{i})d_{j'}
			(g_{i})\sigma(E_{jr})\sigma(E_{j'r'}) 
		\end{align*}
		Now using equations (\ref{key eqn 1}), (\ref{key eqn 2}) and Proposition \ref{prop of f , g exists} we have 
		\begin{align}
			\dis{\sum_{r=1}^{k}} 	\rho(ff_{i}d_{r})\rho(g_{i}d_{r'}) &=
			\begin{cases}
				f\left( \de_{rs}\sigma(E_{s'r'}) +\de_{rs'}\sigma(E_{sr'}) - \sigma(E_{sr}) \sigma(E_{s'r'}) - \sigma(E_{s'r})\sigma(E_{sr'})   \right), \, \text{if} \, s \neq s'  \\
				f\left( \de_{rs}\sigma(E_{sr'}) - \sigma(E_{sr}) \sigma(E_{sr'})   \right), \, \text{if} \, s = s'.
			\end{cases} 
		\end{align} 
	 Hence from the above discussion it follows that $\sigma|_{(W_n)_0}$ satisfied the properties of Lemma \ref{Lemma auxiliary}. Hence $\sigma|_{(W_n)_0}$ is a highest weight module with highest weight $\mu_k$ for some $0 \leq k \leq n$.\\
		{\bf Case II:} Let $X_n=S_n$. 	For any given $s_1,s_2, s_1', s_2' \in \{1,2,\cdots,n\}$, set $\ga=\epsilon_{s_1}+\epsilon_{s_2}+\epsilon_{s_1'}+\epsilon_{s_2'}$. Then there exists $f_1, \cdots, f_k$ and $g_1, \cdots , g_k$ such that Proposition \ref{prop of f , g exists} holds. Now we consider the following action for $i,j,s,t \in \{1,2, \dots, n\}$ 
		$$ \dis{\sum_{r=1}^{k}} 	\rho(d_{ij}(ff_r))\rho(d_{st}(g_r))  $$
		and following the computation of \cite[Theorem 4.3(i)]{YB} together with \cite[Lemma 4.4]{YB}, we deduce that the $\sigma|_{(S_n)_0}$ has to be a highest weight represenation with highest weight $\mu_k$, for some $0 \leq k\leq n-1$.\\
			{\bf Case III:} Let $X_n=S_n$. 	For any given $s_1,s_2, s_1', s_2' \in \{1,2,\cdots,n\}$, set $\ga=\epsilon_{s_1}+\epsilon_{s_2}+\epsilon_{s_1'}+\epsilon_{s_2'}$. Then there exists $f_1, \cdots, f_k$ and $g_1, \cdots , g_k$ such that Proposition \ref{prop of f , g exists} holds. Now we consider the following action:
		$$ \dis{\sum_{r=1}^{k}} 	\rho(d_{H}(ff_r))\rho(d_{H}(g_r))  $$
		and using the computation of \cite[Theorem 3.4(i)]{YB2} together with \cite[Proposition 2.8]{YB2}, we conclude that the $\sigma|_{(H_n)_0}$ has to be a highest weight represenation with highest weight $\mu_k$, for some $0 \leq k\leq m$. 
	This completes the proof. 	
	\end{proof}

	{
		\begin{theorem}\label{Thm irr thm}
			Let us consider the $ \mf L(\mf g,X_n)$ module $A_n \ot V  $ for some irreducible $\mf L(\mf g,X_n)_0$ module $V$. Then  
			\begin{itemize}
				\item[(1)] $A_n \ot V  $ is irreducible if $V \not \cong L^0_X(0, \mu_k, \mbf 0)$, for some $0 \leq k \leq n_X$.
				\item[(2)] For $V \cong L^0_W(0, \mu_n, \mbf 0)$, $A_n \ot V$ is irreducible.
				\item[(3)]  For $V \cong L^0_X(0, \mu_k, \mbf 0)$, $0 \leq k \leq n_X$, $A_n \ot V$ is reducible, except $k=n$.
				
			\end{itemize}
			In particular, if ${\rm dim} V$ is infinite then $A_n \ot V$ is always irreducible $\mf L(\mf g, X_n)$ module. Further $\mcal V_X(\la,\mu,\mbf c)$ reducible if and only if $\mcal V_X(\la,\mu,\mbf c)$ is exceptional and $\mu\neq \mu_n$. 
	\end{theorem} }
	\begin{proof}
Proof of (1),	by Example 3.2, we can consider $A_n \ot V$ as a module in $(A_n, \mf L(\mf g,X_n))$-mod. Now note that $A_n \ot V = \dis{\bgop_{r \in \N^n} t^r \ot V}$. It is clear from the hypothesis that as $\sigma_X(\mf L(\mf g,X_n)) $ module $A_n \ot V$ is completely reducible with none of its irreducible components are isomorphic with $L^0(0, \mu_k, \mbf 0)$, for some $0 \leq k \leq n_X$. Hence any $\mf L(\mf g,X_n)$ submodule $M'$ of $A_n \ot V$ is a $(A_n, \mf L(\mf g,X_n))$ submodule, by Theorem  \ref{Thm for irr}.\\
		Let $M'$ be a non-zero submodule of $A_n \ot V$ as $\mf L(\mf g,X_n)$ module. Then $M'$ contains $1 \ot V$, by Lemma \ref{Lemma contain V }.  Hence $t^r \ot V \subseteq M'$ for all $r \in \N^n$, since $M'$ is a $A_n$ submodule. Therefore $M'=A_n \ot V$. This completes the proof of (1). Note that (2), (3) follows from Proposition \ref{pro V(0,mu,0) red}, Proposition \ref{pro V(0,mu,0) S} and Proposition \ref{pro V(0,mu,0) H}.
	\end{proof}
	
	\section{Character formulas in category $\mcal O$}
	\subsection{Irreducible characters}
	Recall that the triangular decomposition of $\mf g$ and $(X_n)_0$ are given by  $\mf g=\mf n^- \op \mf h \op \mf n^+$ and $(X_n)_0=\mf n^-_X \op \mf h_X \op \mf n^+_X$ respectively. 
	Let $\Phi^+_X$ be the positive root system of $\mf g \op (X_n)_0$.  Moreover, denote by $\Phi^{\geq 1}_X$ be the root system of $\mf L(\mf g,X_n)_{\geq 1}$ relative to $\mf H =\mf h \op \mf h_X$, i.e, $\Phi^{\geq 1}_X= \{\al \in \mf H^*: (\mf L(\mf g,X_n)_{\geq 1})_\al \neq 0\}  $, where
	$$(\mf L(\mf g,X_n)_{\geq 1})_{\al}=\{x \in \mf L(\mf g,X_n)_{\geq 1}: [h,x]=\al(h) x, \, \forall \, h \in \mf H\}.$$
	Then we have $\mf L(\mf g,X_n)_{\geq 1}=\dis{\sum_{\al \in \Phi_X^{\geq 1} }} (\mf L(\mf g,X_n)_{\geq 1})_\al. $  For $\la \in \La^+_{\mf g, X, \C^n},$ we associate a subset of $\mf H^*,$ defined by $D(\lambda)=\{\la_1\in \mf H^* : \la_1 \succeq \lambda\}$, where $\la_1 \succeq \lambda$ means that $\la_1-\lambda$ lies in the $\Z_{\geq 0}$-span of $\Phi^{\geq 1}_X\cup \Phi^+_X$. Now we define an $\C$-algebra $\mcal B$ whose elements are series of the form $\dis\sum_{\lambda\in {\mf H}^*}c_\lambda e^\lambda$ with $c_\lambda\in \C$ and $c_\lambda=0$ for $\lambda$ outside the union of a finite number of sets of the form $D(\la)$.
	Consequently $\mcal B$ becomes a commutative associative algebra if we define $e^\la e^\mu = e^{\la+\mu}$ with $e^0$ as identity element. It is well known that formal exponents $\{e^\la : \la \in \mf H^*\}$ are linearly independent and in one to one correspondence with $\mf H^*$. Let $W$ be a semi-simple $\mf H$-module with finite dimensional weight spaces such that $W=\dis{\sum_{\la \in \mf H^*} }W_\la. $ Then we define ${\rm ch}(W)=\dis \sum_{\lambda\in {\mf H}^*}\dim W_\lambda e^\lambda$. In particular, if $V$ is an object of $\mathcal{O}$ for $\mf L(\mf g, X_n)$, then ${\rm ch}(V)\in \mathcal{B}$. The following statements are standard.

	\begin{lemma}\label{Lem Char} 
		\begin{itemize}
			\item[(1)] Let $V_1, V_2$ and $V_3$ be three $\mf L(\mf g,X_n)$-modules in the category $\mcal O$ such that there is an exact sequence of $\mf L(\mf g,X_n)$-modules $0\rightarrow V_1\rightarrow V_2\rightarrow V_3\rightarrow 0 $. Then ${\rm ch}(V_2)=\ch(V_1)+\ch(V_3)$.
			\item[(2)] Let $W=\dis\sum_{\lambda \in \mf H^*}W_\lambda$ be a semi-simple $\mf H$-module with finite-dimensional weight spaces, and $U=\dis\sum_{\lambda \in \mf H^*}U_\lambda$ be a finite-dimensional $\mf H$-module. If $\ch(W)=\dis\sum_{\la \in \mf H^*}c_{\lambda}e^\lambda$ falls in $\mathcal{B}$, then $\ch(W\otimes U)$ must fall in $\mathcal{B}$ and $\ch(W\otimes U)=\ch(W)\ch(U)$.
		\end{itemize}
	\end{lemma}
	
	Now we study formal character for $\Delta_X(\la,\mu, \mbf c)$, for some  $(\la,\mu, \mbf c) \in \La^+_{\mf g,X, \C^n}$. Since $\Delta_X(\la,\mu,\mbf c)= U(\mf L(\mf g,X_n)_{\geq 1}) \ot_{\C} L^0_X(\la,\mu,\mbf c)$, so as a $U(\mf L(\mf g,X_n)_{\geq 1})$-module $\Delta_X(\la,\mu,\mbf c)$ is free module of rank dim $L^0_X(\la,\mu,\mbf c)$ and generated by $L^0_X(\la,\mu,\mbf c)$. Hence by Lemma \ref{Lem Char}, $$\ch(\Delta_X(\la,\mu,\mbf c))=\ch( U(\mf L(\mf g,X_n)_{\geq 1})) \ch(L^0_X(\la,\mu,\mbf c)).$$ Set
	\begin{align}
		\Upsilon_X=\prod\limits_{\alpha\in\Phi^{\geq 1}_X}
		(1-e^{\alpha})^{- dim({\mf L(\mf g,X_n)_{\geq 1}})_\alpha}.
	\end{align}
	Then we have $\ch(\Delta_X(\lambda,\mu,\mbf c))=\Upsilon_X\ch L^0_X(\lambda,\mu,\mbf c)$.
	
	\subsubsection{\bf Character formula for $L_X(\la,\mu,\mbf c)$} First note that there is no  natural action of $\mf L(\mf g,X_n)$ on $A_n$. But we can define an action of $\mf L(\mf g,X_n)$ on $A_n$ which makes it a module of Category $\mcal O$. In particular, define the action of $X_n$ on $A_n$ as natural action induced from the action of $W_n$ and action of $\mf g\ot A_n \op \mcal Z$ as trivial. It is easy to see that this action define a module structure of $\mf L(\mf g,X_n)$ on $A_n$. Further, under this action $A_n \in \mcal O$ and each of its weight spaces are one dimensional. In particular, $A_n=\dis{\sum_{r \in \N^n}} \C t^r$ and $\C t^r =(A_n)_{\xi_r} $, where $(A_n)_{\xi_r}$ is the weight space corresponding to the weight 
$$ \xi_r=	\begin{cases}
    \dis{\sum_{i=1}^{n}}r_i\epsilon_i, \,\, \text{for } X_n=W_n,\\
     \dis{\sum_{i=1}^{n-1}}(r_i-r_n)\epsilon_i, \,\, \text{for } X_n=S_n,\\
      \dis{\sum_{i=1}^{m}}(r_i-r_{m+i})\epsilon_i, \,\, \text{for } X_n=H_n.
	\end{cases}$$
Thus we have $\ch(A_n)=\dis{\sum_{r \in \N^n}}e^{\xi_r}$. Let us denote  $\Gamma_X= \ch(A_n)$ for $X_n=W_n, S_n$ or $H_n$. Then by Lemma \ref{Lem Char}, we have $\ch (\mcal V_X(\la,\mu,\mbf c))= \ch(A_n \ot L^0_X(\la,\mu,\mbf c))= \Gamma_X \ch(L^0_X(\la,\mu,\mbf c))$. Now we are ready to compute the character formulas for irreducible modules in category $\mcal O$.
	\begin{theorem}\label{thm: irr char}
		 Then we have the following character formulas for irreducible modules of category $\mcal O$.
		\begin{itemize}
			\item[(1)] \ch$(L_X(\la,\mu,\mbf c))	=\Gamma_X \ch(L^0_X(\la,\mu,\mbf c))$, for $(\la,\mu,\mbf c) \neq (0,\mu_k,0), $  $0 \leq k \leq n_X$ and $k \neq n$.
			\item[(2)] $\ch(L_W(0,\mu_{k},\mbf 0))=\dis{\sum_{i=0}^{k-1} } (-1)^{k-i+1}\Gamma_W \ch(L^0_W(0,\mu_{i},\mbf 0)) +(-1)^ke^0,  \,\,\, \forall \, 0 \leq k \leq n-1.$
			\item[(3)] $\ch(L_S(0,\mu_{k},\mbf 0))=\dis{\sum_{i=0}^{k-1} } (-1)^{k-i+1}\Gamma_S \ch(L^0_S(0,\mu_{i},\mbf 0)) +(-1)^ke^0,  \,\,\, \forall \, 0 \leq k \leq n-1.$
			\item[(3)] $\ch(L_H(0,\mu_{k},\mbf 0))=\dis{\sum_{i=0}^{k-1} } (-1)^{k-i+1}(k-i)\Gamma_H \ch(L^0_H(0,\mu_{i},\mbf 0)) +(-1)^{k}(k+1)e^0,  \,\,\, \forall \, 0 \leq k \leq m.$
		\end{itemize}
\end{theorem}	

	\begin{proof}
		To prove (1) note that for $(\la,\mu,\mbf c) \neq (0, \mu_k,0)$ with $0 \leq k \leq n_X$ and $k \neq n$, we have $\mcal V_X(\la, \mu, \mbf c)$ is irreducible, by Theorem \ref{Thm irr thm}. Hence by Lemma \ref{description of socle}, we get $\mcal V_X(\la,\mu,\mbf c) \cong L_X(\la,\mu,\mbf c)$ and the result (1) follows. We know that
		\begin{align}\label{eqn for char for}
			\ch(\mcal V_X(\la,\mu,\mbf c))=\dis{\sum_{(\la',\mu',\mbf c')}}[\mcal
			V_X(\la,\mu,\mbf c): L_X(\la',\mu',\mbf {c'})]\ch(L_X(\la',\mu',\mbf {c'})).  
		\end{align} 
		Proof of (2), note that we get the following inductive formula with the help of Theorem \ref{Thm irr thm}, Proposition \ref{pro V(0,mu,0) red} and equation (\ref{eqn for char for}).
		\begin{align*}
			\ch(L_W(0,\mu_{k+1},\mbf 0))= \Gamma_W \ch(L^0_W(0,\mu_{k},\mbf 0)) - \ch(L_W(0,\mu_{k},\mbf 0)),
		\end{align*}
	 for all $ 0 \leq k \leq n-2$ and  $\ch(L_W(0,\mu_0,0))= e^0$. Now from the above inductive formula it follows that,
		$$\ch(L_W(0,\mu_{k},\mbf 0))=  \dis{\sum_{i=0}^{k-1}}(-1)^{k-i+1}\Gamma_W \ch(L^0_W(0,\mu_{i},\mbf 0)) +(-1)^{k}e^0 , \,\, 0 \leq k \leq n-1.  $$
This completes the proof of (2), and proof of (3) follows with a similar arguments  by Theorem \ref{Thm irr thm} and Proposition \ref{pro V(0,mu,0) S}.\\
We prove (4) by induction on $k$. First note that the result is true for $k=0$, since $L_H(0, \mu_0, \mbf 0)$ is trivial one dimensional module. Now assume that the result (4) holds for all  $0 \leq k \leq r$ for some $0 \leq r \leq m-1$. Now by Proposition \ref{pro V(0,mu,0) H} and equation (\ref{eqn for char for}), we have
\begin{align*}
	\ch(L_H(0,\mu_{r+1},\mbf 0)) &= \Gamma_H \ch(L^0_H(0,\mu_{r},\mbf 0)) - 2\ch(L_H(0,\mu_{r},\mbf 0)) -\ch(L_H(0,\mu_{r-1},\mbf 0)) \cr
	&= \Gamma_H \ch(L^0_H(0,\mu_{r},\mbf 0))-2 (\dis{\sum_{i=0}^{r-1} } (-1)^{r-i+1}(r-i)\Gamma_H \ch(L^0_H(0,\mu_{i},\mbf 0)) +(-1)^{r}(r+1)e^0) \cr
	&  -(\dis{\sum_{i=0}^{r-2} } (-1)^{r-i}(r-1-i)\Gamma_H \ch(L^0_H(0,\mu_{i},\mbf 0)) +(-1)^{r-1}re^0) \cr
	&= \Gamma_H \ch(L^0_H(0,\mu_{r},\mbf 0)) -2\Gamma_H \ch(L^0_H(0,\mu_{r-1},\mbf 0)) +(-1)^{r+1}(r+2)e^0 \cr 
	& + \dis{\sum_{i=0}^{r-2} } (-1)^{r-i}(r+1-i)\Gamma_H \ch(L^0_H(0,\mu_{i},\mbf 0) \cr
	&= \dis{\sum_{i=0}^{r} } (-1)^{r-i}(r+1-i)\Gamma_H \ch(L^0_H(0,\mu_{i},\mbf 0) +(-1)^{r+1}(r+2)e^0.
\end{align*} 
This proves that the induction hypothesis holds for $k=r+1$. Hence completes the proof.

	\end{proof}

	\subsection{Tilting modules and character formulas}\label{sec: tilt sec}
	In this section, we determine the character formulas for tilting modules in $\mathcal{O}$ by following the results of Soergel, \cite{Soe}. We need some preparation to use Soergel's result. We start with the following.
	\subsubsection{\bf Semi-infinite character}
	\begin{definition}
		Let $\mf L=\displaystyle{\sum_{i \in \Z}}\mf L_i$ be a $\Z$ graded Lie algebra with dim $\mf L_i < \infty$ for all $i \in \Z$. A character $\gamma: \mf L_0 \to \C$ is said to be a semi-infinite character for $\mf g$ if it satisfies the following properties.
		\begin{itemize}
			\item[(SI-1)] As a Lie algebra, $\mf L$ is generated by $\mf L_1, \mf L_0$ and $ \mf L_{-1}$.\\
			\item[(SI-2)] $\gamma([X,Y])=tr(\ad X\circ \ad Y \mid_{\mf L_0})$, for all $X \in \mf L_1$ and $Y \in \mf L_{-1}$.
		\end{itemize}
	\end{definition}	

	\begin{lemma}\label{semi-inf}
	\begin{itemize}
		\item[(1)] 	A linear map $\mcal E_W $ is a semi-infinite character on $\mf L(\mf g,W_n)$ if and only if $\mcal E_W(t_id_j)=\delta_{ij}$,  $\mcal E_W(x)=0$,  for all $1 \leq i,j \leq n$ and $x \in \mf g \op \mcal Z(\mf L(\mf g))$. Consequently, $\mcal E_W|_{\mf h_W}=\mu_n$. 
		\item[(2)] A linear map $\mcal E_X $ is a semi-infinite character on $\mf L(\mf g,X_n)$ if and only if $\mcal E_X=0$ for $X_n= S_n,H_n$.
	\end{itemize}
	\end{lemma}
	\begin{proof}
	Proof of both the statements goes parallelly and we write it for all $X_n$ together. It follows from the proof of \cite[Lemma 2.1]{DSY20} that $\mcal E_X|_{ ({X}_n)_0}$ is a semi-infinite character for $X_n$. In particular, we get that $X_n$ is generated by $(X_n)_{-1} , (X_n)_0 $ and $(X_n)_1$ with the required property stated in the Lemma, i.e $\mcal E_W(t_id_j)=\de_{ij}$ and $\mcal E_X|_{(X_n)_0}=0$ for $X_n=S_n,H_n$. Now note that $[(X_n)_0,  \mf g \op \mcal Z(\mf L(\mf g))]=[\mf L(\mf g,X_n)_{-1},  \mf g \op \mcal Z(\mf L(\mf g,X_n))]=0$. From this observation it will follow that $\mcal E_X(x)=0$ for all $x \in  \mf g \op \mcal Z(\mf L(\mf g))$, by the help of (SI-2) property.\\
Now	to complete the proof we need to show that $\mf L(\mf g, X_n)$ is generated by $\mf L(\mf g, X_n)_{-1}, \mf L(\mf g, X_n)_0$ and $\mf L(\mf g, X_n)_1$. Now the result follows from the following observation and the fact that $X_n$ is generated by $(X_n)_{-1} , (X_n)_0 $ and $(X_n)_1$.

$$  [t^rd_i,x \ot t_i]=x\ot t^r  \,\,\, \text{and} \,\,\, [t^rd_i, t_iK_j]=t^rK_j; \,\,   $$
$$[d_{ij}(t^{r+e_j}),x\ot t_i] =(r_j+1)x\ot t^r  \,\,\, \text{and} \,\,\, [d_{ij}(t^{r+e_j}), t_iK_j]=(r_j+1) t^rK_j;   $$
$$ [d_{H}(t^{r+e_j}),x\ot t_{m+j}] =(r_j+1)x\ot t^r \,\,\, \text{and} \,\,\, [d_{H}(t^{r+e_j}),t_{m+j}K_s]=(r_j+1) t^rK_s, $$
$\forall r \in \N^n, \, 1\leq j,s \leq n, \, x\in \mf g. $

	\end{proof}
	
	\subsubsection{\bf Tilting modules}  
\begin{definition}
	An object $M\in\mathcal{O}$ is said to admit a $\Delta$-flag if there exists an increasing filtration
	$$0=M_0\subset M_1\subset M_2\subset\cdots\cdots$$
	such that $M=\bigcup\limits_{i=0}^{\infty}M_i$ and $M_{i+1}/M_{i}\cong \Delta_X(\lambda_i)$ for all $i\geq 1$  where $\lambda_i\in\Lambda^+_{\mf g,X,\C^n}$ for all $i$.
\end{definition}

According to Soergel's construction (see \cite{Soe}), we get that associated with each $(\lambda,\mu,\mbf c)\in \Lambda^+_{\mf g, X,\C^n}$, there exists a uniquely determined indecomposable tilting module $T_X(\lambda,\mu,\mbf c)$  whose $\Delta$-flag starts with $\Delta_X(\lambda,\mu,\mbf c)$. The following result follows with similar lines of \cite[Theorem 5.2]{Soe}.

\begin{proposition}\label{tilting}
	For each $(\lambda,\mu,\mbf c)\in\Lambda^+_{\mf g,X, \C^n}$, up to isomorphism, there exists a unique indecomposable object $T_X(\lambda,\mu,\mbf c)\in\mathcal{O}$ such that
	\begin{itemize}
		\item[(1)] $Ext_{\mathcal{O}}^1(\Delta_X(\la',\mu',\mbf c'), T_X(\lambda,\mu,\mbf c))=0,\,\forall\,(\la',\mu',\mbf c')\in\Lambda^+_{\mf g,X, \C^n}$.
		\item[(2)] $T_X(\lambda,\mu,\mbf c)$ admits a $\Delta$-flag, starting with $\Delta_X(\lambda,\mu,\mbf c)$ at the bottom.
	\end{itemize}
\end{proposition}

We call the indecomposable module $T_X(\lambda,\mu,\mbf c)$ as the tilting module corresponding to $(\lambda,\mu,\mbf c)\in\Lambda^+_{\mf g,X, \C^n}$.
Now we get the following formula by Theorem \cite[Theorem 5.12]{Soe} together with the help of Lemma \ref{semi-inf}.
\begin{proposition} \label{soegrel formula}
	Let $(\lambda, \mu,\mbf c), (\lambda', \mu',\mbf c') \in\Lambda^+_{\mf g,X, \C^n}.$ 
	Then we have
	$$[T_X(\lambda,\mu,\mbf c):\Delta_X(\la',\mu',\mbf c')]=[\mathcal{V}_X(-\omega_{\mf g} \la',-\omega_{X}\mu'-\mcal E_X|_{\mf h_X}, -\mbf {c'}):
	L_X(-\omega_{\mf g} \la,-\omega_{X}\mu-\mcal E_X|_{\mf h_X}, -\mbf {c})],$$
	where $\omega_{\mf g}$ and $\omega_{X}$  are the longest elements in the Weyl group of $\mf g$ and $(X_n)_0$ respectively.
	
\end{proposition}
\begin{proof}
	The proof of this proposition follows from the same lines with \cite[Theorem 5.12]{Soe} and Remark 2.2 of \cite{Soe}, together with the observation that highest weight of the irreducible module $L^0(\la,\mu,\mbf c)^*$ is $(-\omega_{\mf g}\la, -\omega_X
	 \mu)$ and $K_i$ acts on $  L^0_X(\la,\mu,\mbf c)^*$ by $-c_i$ for all $1 \leq i \leq n$. 
\end{proof}
\begin{remark}\label{weyl act}
	It should be noted that
	\begin{equation*}
		\omega_X(\epsilon_i)= \begin{cases}
		\epsilon_{n+1-i}, \,\,\, \, \text{for } X_n=W_n \text{ or } S_n\\
		-\epsilon_i, \,\,\, \, \text{for } X_n=H_n,
		\end{cases}
	\end{equation*}
for all	$ 1 \leq i \leq n. \,$
\end{remark}

Now we will determine the multiplicity of the standard module $\Delta_X(\la',\mu,\mbf c')$ occurring in the $\Delta$-filtration of the tilting module $T_X(\lambda,\mu,\mbf c)$ for any $(\lambda,\mu,\mbf c),(\lambda',\mu',\mbf c'), \in\Lambda^+_{\mf g,X} \times \C^n$.

\begin{proposition}\label{prop mult Delta(lam)}
	Let $(\lambda',\mu',\mbf c'),$ $(\lambda,\mu,\mbf c)\in\Lambda^+_{\mf g,W} \times \C^n$. Then the following statements hold.
	\begin{itemize}
		\item[(1)] If $(\la',\mu',\mbf c')=(0,-2\sum\limits_{i=1}^k\epsilon_{n+1-i}-\sum\limits_{i=k+1}^n\epsilon_{n+1-i},\mbf 0) $ for some $k$ with $0\leq k \leq n-1$. Then $[T_W(\lambda,\mu,\mbf c):\Delta_W(\la',\mu',\mbf c')] \neq 0 $ if and only if $\la=0$, $\mbf c=\mbf 0$ and  $\mu=\mu'$ or $\mu=\mu'-\epsilon_{n-k}$. Moreover, $[T_W(0,\mu',\mbf 0):\Delta_W(0,\mu',\mbf 0)]=[T_W(0,\mu'-\epsilon_{n-k}, \mbf 0):\Delta_W(0,\mu',\mbf 0)]=1$.
		\item[(2)] If $(\la',\mu',\mbf c')\neq (0,-2\sum\limits_{i=1}^k\epsilon_{n+1-i}-\sum\limits_{i=k+1}^n\epsilon_{n+1-i},\mbf 0) $ for some $k$ with $0\leq k \leq n-1$. Then $[T_W(\lambda,\mu,\mbf c):\Delta_W(\la',\mu',\mbf c')] \neq 0 $ if and only if $\la=\la'$, $\mbf c=\mbf c'$ and  $\mu=\mu'$. Moreover, $[T_W(\la,\mu,\mbf c):\Delta_W(\la,\mu,\mbf c)]=1$.
		
	\end{itemize}
\end{proposition}

\begin{proof}
	By Proposition \ref{soegrel formula} and Lemma  \ref{semi-inf} we have,
	 $$[T_W(\lambda,\mu,\mbf c):\Delta_W(\la',\mu',\mbf c')]=[\mathcal{V}_W(-\omega_{\mf g} \la',-\omega_{W}\mu'-\mu_n, -\mbf {c'}):
	L_W(-\omega_{\mf g} \la,-\omega_{W}\mu-\mu_n, -\mbf {c})].$$
	Now we claim that when $\mathcal{V}_W(-\omega_{\mf g} \la',-\omega_{W}\mu'-\mu_n, -\mbf {c'})$ not simple, then $(\la',\mu',\mbf c') $ satisfy the hypothesis of (1). Suppose $\mathcal{V}_W(-\omega_{\mf g} \la',-\omega_{W}\mu'-\mu_n, -\mbf {c'})$ not simple. Then by Proposition \ref{pro V(0,mu,0) red}, we have $\la' =0, \mbf c'=\mbf 0$ and $\mu'$ satisfy the relation $-\omega_W \mu'-\mu_n=\mu_k=\epsilon_1 \cdots + \epsilon_k,$ for some $0 \leq k \leq n-1$. From this we have $\mu'=-2\sum\limits_{i=1}^k\epsilon_{k+1-i}-\sum\limits_{i=k+1}^n\epsilon_{n+1-i}$, this proves the claim.\\
	 Again from Proposition \ref{pro V(0,mu,0) red}, we have $\mcal V_W(0, \mu_k,\mbf 0)$ have two composition factors $L_W(0,\mu_k,\mbf 0)$ and $L_W(0,\mu_{k+1},\mbf 0)$. Hence we get $\la=0, \mbf c=0$ and $-\omega_W\mu-\mu_n = \mu_k $ or $\mu_{k+1}$, i.e $\mu =\mu'$ or $\mu =\mu'-\epsilon_{n-k}. $ Now the result (1) follows from Proposition \ref{pro V(0,mu,0) red}.\\
	From the proof of (1) it is clear that when $(\la',\mu',\mbf c')$ satisfied the condition of (2), then $\mathcal{V}_W(-\omega_{\mf g} \la',-\omega_{W}\mu'-\mu_n, -\mbf {c'})$ is simple. Hence it has only one composition factor, therefore we have $\la'=\la, \mu'=\mu, \mbf c'=\mbf c$. This completes the proof of (2).
\end{proof}
By similar arguments of Proposition \ref{prop mult Delta(lam)} we have the following proposition. 
 \begin{proposition}\label{prop mult Delta(lam)}
 	Let $(\lambda',\mu',\mbf c'),$ $(\lambda,\mu,\mbf c)\in\Lambda^+_{\mf g,X} \times \C^n$, for $X_n=S_n $ or $H_n$. Then the following statements hold.
 	\begin{itemize}
 		\item[(1)] If $(\la',\mu',\mbf c')=(0,-\sum\limits_{i=1}^k\epsilon_{n+1-i},\mbf 0) \in\Lambda^+_{\mf g,S} \times \C^n$ for some $k$ with $0\leq k \leq n-1$. Then $[T_S(\lambda,\mu,\mbf c):\Delta_S(\la',\mu',\mbf c')] \neq 0 $ if and only if $\la=0$, $\mbf c=\mbf 0$ and  $\mu=\mu'$ or $\mu=\mu'-\epsilon_{n-k}$. Moreover, $[T_S(0,\mu',\mbf 0):\Delta_S(0,\mu',\mbf 0)]=[T_S(0,\mu'-\epsilon_{n-k}, \mbf 0):\Delta_S(0,\mu',\mbf 0)]=1$.\\
 		
 		\item[(2)] If $(\la',\mu',\mbf c')\neq (0,-\sum\limits_{i=1}^k\epsilon_{n+1-i},\mbf 0) \in\Lambda^+_{\mf g,S} \times \C^n$ for some $k$ with $0\leq k \leq n-1$. Then $[T_S(\lambda,\mu,\mbf c):\Delta_S(\la',\mu',\mbf c')] \neq 0 $ if and only if $\la=\la'$, $\mbf c=\mbf c'$ and  $\mu=\mu'$. Moreover, $[T_S(\la,\mu,\mbf c):\Delta_S(\la,\mu,\mbf c)]=1$.
 			\item[(3)] If $(\la',\mu',\mbf c')=(0,\mu_k,\mbf 0) \in\Lambda^+_{\mf g,H} \times \C^n $ for some $k$ with $0\leq k \leq m$. Then $[T_H(\lambda,\mu,\mbf c):\Delta_H(\la',\mu',\mbf c')] \neq 0 $ if and only if $\la=0$, $\mbf c=\mbf 0$ and $\mu=\mu'-\epsilon_{k}$ or $\mu=\mu'$ or $\mu=\mu'+\epsilon_{k+1}$. Moreover, $[T_H(0,\mu'-\epsilon_k,\mbf 0):\Delta_H(0,\mu',\mbf 0)]=[T_H(0,\mu'+\epsilon_{k+1}, \mbf 0):\Delta_H(0,\mu',\mbf 0)]=1$ and $[T_H(0,\mu', \mbf 0):\Delta_H(0,\mu',\mbf 0)]=2$\\
 		
 		\item[(4)] If $(\la',\mu',\mbf c')\neq (0,\mu_k,\mbf 0) \in\Lambda^+_{\mf g,H} \times \C^n$ for some $k$ with $0\leq k \leq m$. Then $[T_H(\lambda,\mu,\mbf c):\Delta_H(\la',\mu',\mbf c')] \neq 0 $ if and only if $\la=\la'$, $\mbf c=\mbf c'$ and  $\mu=\mu'$. Moreover, $[T_H(\la,\mu,\mbf c):\Delta_H(\la,\mu,\mbf c)]=1$.
 		
 	\end{itemize}
 \end{proposition}

The following theorem is a consequence of Proposition \ref{prop mult Delta(lam)} and Proposition \ref{soegrel formula} and equation (4.1).
\begin{theorem}
	The following statements hold. 
	\begin{itemize}
		\item[(1)] If $(\la,\mu,\mbf c)=(0,-2\sum\limits_{i=1}^k\epsilon_{n+1-i}-\sum\limits_{i=k+1}^n\epsilon_{n+1-i},\mbf 0) \in\Lambda^+_{\mf g,W} \times \C^n $ for some $k$ with $0\leq k \leq n-1$. Then 
		$$\ch(T_W(0,\mu,\mbf 0))=\Upsilon_W(\ch(L^0_W(0,\mu,\mbf 0))+\ch(L^0_W(0, \mu+\epsilon_{n-k},\mbf 0))) .$$
		
		\item[(2)]	If $(\la,\mu,\mbf c)\neq (0,\sum\limits_{i=1}^k\epsilon_{n+1-i}-\sum\limits_{i=k+1}^n\epsilon_{n+1-i},\mbf 0) \in\Lambda^+_{\mf g,W} \times \C^n $ for some $k$ with $0\leq k \leq n-1$. Then 
		$$\ch(T_W(\la,\mu,\mbf c))=\Upsilon_W(\ch(L^0_W(\la,\mu,\mbf c)).$$ 
			\item[(3)] If $(\la,\mu,\mbf c)=(0,-\sum\limits_{i=1}^k\epsilon_{n+1-i},\mbf 0)  \in\Lambda^+_{\mf g,S} \times \C^n $ for some $k$ with $0\leq k \leq n-1$. Then 
		$$\ch(T_S(0,\mu,\mbf 0))=\Upsilon_S(\ch(L^0_S(0,\mu,\mbf 0))+\ch(L^0_S(0, \mu+\epsilon_{n-k},\mbf 0))) .$$
		
		\item[(4)]	If $(\la,\mu,\mbf c)\neq (0,-\sum\limits_{i=1}^k\epsilon_{n+1-i},\mbf 0) \in\Lambda^+_{\mf g,S} \times \C^n $ for some $k$ with $0\leq k \leq n-1$. Then 
		$$\ch(T_S(\la,\mu,\mbf c))=\Upsilon_S(\ch(L^0_S(\la,\mu,\mbf c)).$$
		\item[(5)] If $(\la,\mu,\mbf c)=(0,\mu_k,\mbf 0)  \in\Lambda^+_{\mf g,H} \times \C^n $ for some $k$ with $0\leq k \leq m$. Then 
		$$\ch(T_H(0,\mu,\mbf 0))=\Upsilon_H(\ch(L^0_H(0,\mu+\epsilon_k,\mbf 0))+2\ch(L^0_H(0,\mu,\mbf 0))+\ch(L^0_H(0, \mu-\epsilon_{k+1},\mbf 0))) .$$
		
		\item[(6)]	If $(\la,\mu,\mbf c)\neq (0,\mu_k,\mbf 0) \in\Lambda^+_{\mf g,H} \times \C^n $ for some $k$ with $0\leq k \leq m$. Then 
		$$\ch(T_H(\la,\mu,\mbf c))=\Upsilon_H(\ch(L^0_H(\la,\mu,\mbf c)),$$
	\end{itemize}
where $\Upsilon_X =\ch(\Delta_X(\la,\mu,\mbf c))$.
\end{theorem}

\vspace{5cm} 
{\bf Acknowledgments:} 
The author would to like to thank Institute postdoctoral Fellowship of IISc Bangalore for funding this research.

\end{document}